\documentclass[a4paper]{amsart}

\usepackage{graphicx}
\usepackage{amssymb}
\usepackage{amsmath}
\usepackage{bbm}
\usepackage{stmaryrd}
\usepackage{comment}
\usepackage{ifthen}

\newtheorem{lemma}{Lemma}
\newtheorem{corollary}{Corollary}
\newtheorem{theorem}{Theorem}
\newtheorem{proposition}{Proposition}

\newtheorem{othertheorem}{Theorem}

\theoremstyle{definition}
\newtheorem{definition}{Definition}

\newcommand{\ISE}{{\scriptscriptstyle{\mathrm{ISE}}}}

\author[Guillaume Chapuy]{Guillaume Chapuy $^1$}
\thanks{$^1$Laboratoire d'Informatique de l'\'{E}cole Polytechnique, 91128
Palaiseau Cedex, France. \tt guillaume.chapuy@lix.polytechnique.fr}
\title[The structure of unicellular maps]{The structure of unicellular maps, 
and a~connection~between~maps~of~positive~genus and~planar labelled trees.}
\begin{document}

\begin{abstract}
A unicellular map is a map which has only one face. We give a
bijection between a dominant subset of rooted unicellular maps of given genus
and a set of rooted plane trees with distinguished vertices. The bijection applies as
well to the case of labelled unicellular maps, which are related to all rooted maps by
Marcus and Schaeffer's bijection.

This gives an immediate derivation of the asymptotic number of unicellular maps
of  given genus, and a simple bijective proof of a formula of Lehman and Walsh
on  the number of triangulations with one vertex. From the labelled case, we
deduce   an expression of the asymptotic number of maps of genus $g$ with $n$
edges involving the ISE random measure, and an explicit characterization of the
limiting  profile and radius of random bipartite quadrangulations of genus $g$ 
in terms of the ISE.
\end{abstract}

\maketitle
\date{\today}

\section{Introduction.}

The enumerative study of orientable maps, or graphs embedded on orientable
surfaces, has been an important domain of mathematics since the works of
Tutte in the sixties~(\cite{Tutte:census}).
In the last ten years, many progresses have been made in the
combinatorial and statistical study of \emph{planar maps}, or graphs drawn on the
sphere, thanks to increasingly powerful bijective techniques, originating
in Schaeffer's thesis~(\cite{Sc:PhD}). All these techniques
rely on canonical decompositions of maps into suitable classes of \emph{plane
trees}, which are much more easily amenable to mathematical analysis than maps
themselves (see for example \cite{MBM-Sc,BDFG:census,BDFG}).
Beyond strong enumerative results, these methods enabled to finely describe the
metric properties of large random planar maps (for
example \cite{ChassaingSc,BDFG:distances,Miermont:Geodesiques,BG:3points}),
opening the way to the probabilistic study of a possible limiting continuum object, 
the \emph{Brownian planar map}, related to the continuum random tree and the
Brownian snake (\cite{Marckert-Mokkadem-BrownianMap,LG,LGPa,LG:Geodesiques}).

On the other hand, maps of positive genus (roughly speaking, graphs embedded on
a torus with $g$ handles), have been considered by several authors, mainly from
the enumerative viewpoint. Using an extension of Tutte's method, Bender and
Canfield (\cite{BeCa0}) showed that the number of rooted maps of fixed genus $g$ with $n$ edges is asymptotically
equivalent to $t_g n^{5(g-1)/2}12^n$, when $n$ tends to infinity. Here $t_g$ is
a constant, which plays an important role in theoretical physics and in
geometry (see \cite{LaZv}). It is known, from matrix integrals techniques,
that the numbers $(t_g)_{g\geq 1}$ satisfy remarquable non linear recurrence relations,
related to the {Painlev\'e\nobreakdash-I} equation (\cite{LaZv}, page 201,
or the recent equivalent results in \cite{Goulden-Jackson:tg,Bender-Gao-Richmond:tg}).
There is however, no combinatorial explanation of those properties.

From the bijective side, the
Marcus-Schaeffer bijection (\cite{MaSc}) relates maps of positive genus to
labelled \emph{unicellular maps}, or maps with one face, of the \emph{same
genus}. Thanks to this bijection and its generalizations, it has been possible
recently to re-derive the counting exponent $n^{5(g-1)/2}$ for several families
of rooted maps (\cite{ChMaSc,Chapuy:constellations}), and to exhibit certain metric
properties of large random maps of given genus (\cite{Miermont:Geodesiques}).
However, these approaches are not sufficient to completely describe the combinatorics of maps of genus $g$: a lot of 
the structure is still hidden in
the unicellular maps, which contain, in a sense, all the information specific
to the genus. For example, these methods give the constant
$t_g$ as a sum, over a finite number of cases, of a complicated
combinatorial quantity, and it seems difficult to analyse or to compute
the $t_g$'s in this way.

The purpose of this paper is to ``break the genus'': we give a
bijective construction that relates unicellular maps of fixed genus to 
suitably decorated \emph{plane} trees (precisely, our construction concerns
only a dominating proportion of unicellular maps of given genus, so that 
we obtain mainly \emph{asymptotic} results).
The first consequence of our construction is the bijective derivation of the
asymptotic number of unicellular maps of fixed genus. The result itself is well
known (the numbers are even \emph{exactly} known, see
\cite{Harer-Zagier,Jackson:countingcycles,Goupil-Schaeffer}), but our
derivation is elementary and gives a more constructive viewpoint. Moreover, it
adapts to the case of \emph{labelled} unicellular maps, 
and implies, thanks to the Marcus-Schaeffer
bijection, new relations between maps of positive genus and \emph{plane}
labelled trees. 
For example, we relate the constant $t_g$ to the $g$-th moment
of the random variable $\int_{-\infty}^\infty
f_{{\scriptscriptstyle\mathrm{ISE}}}(x)^3dx$, where $f_\ISE$ is the
density of the random probability measure ISE (\emph{Integrated
Superbrownian Excursion}, which describes the limiting repartition
of labels in a large random labelled tree, see
\cite{Aldous-treebasedmodels,MBM-Janson}). We obtain: $$t_g = \frac{2}{2^{5g/2}g!\sqrt{\pi}}\mathbb{E}\left[\left(\int_{-\infty}^\infty
f_{{\scriptscriptstyle\mathrm{ISE}}}(x)^3dx\right)^g\right]. $$
Observe that, contrarily to
\cite{ChMaSc,Miermont:Geodesiques}, this formula does not
involve any summation nor case disjonction. Moreover, it makes a connection
with the theory of superprocesses, hopefully opening the way to a better
understanding of the sequence $(t_g)_{g\geq 1}$ and its recurrence properties.
Finally, as a last consequence of our bijection, we prove the
convergence of the normalized profile and radius of pointed
quadrangulations of fixed genus, and we characterize their limit
in terms of the ISE.

{\bf Organization of the paper:} in Section~\ref{sec:maps}, we give formal
definitions related to maps; in Section~\ref{sec:scheme}, we give a
description of unicellular maps in terms of \emph{schemes}, and we identify the
dominating cases (this is already contained in \cite{ChMaSc});
Section~\ref{sec:key} contains the key combinatorial lemma, which leads
to the description of the bijection in Section~\ref{sec:opening}; finally,
Section~\ref{sec:labelled} studies the case of labelled unicellular maps and
gives our new expression of $t_g$, while Section~\ref{sec:profile} is devoted to
the convergence of the profile and radius of pointed quadrangulations.

\section{Maps.}

\label{sec:maps}
\subsection{Definitions}

We begin with a combinatorial definition of a map.
\begin{definition}
Let $n$ be a positive integer. A \emph{map of size $n$} is a triple
$\mathfrak{m}=(\alpha,\beta,\gamma)$ of permutations of
$\llbracket 1,2n\rrbracket$, such that:
\begin{itemize}
  \item $\beta \alpha = \gamma$
  \item $\alpha$ is an involution without fixed points (i.e. all its cycles have
  length $2$)
\end{itemize}
The orbits of $\llbracket 1,2n\rrbracket$ under the action of the subgroup of
$\mathfrak{S}_{2n}$ generated by $\alpha$, $\beta$ and $\gamma$ are called the \emph{connected
components} of $\mathfrak{m}$. If this action is transitive, we say that
$\mathfrak{m}$ is \emph{connected}.
\end{definition}
We use the cycle notation for permutations. For example, the permutation
$(1,4,3)(2,$ $5,6)$ of $\llbracket 1,6\rrbracket$ sends $1$ to $4$, $6$ to $2$,
etc\ldots The number of cycles of a permutation $\sigma$ is denoted $|\sigma|$.
A map $\mathfrak{m}=(\alpha, \beta, \gamma)$ being given, the cycles of
$\alpha$, $\beta$ and $\gamma$ are called its \emph{edges}, \emph{vertices} and
\emph{faces}, respectively. The size of $\mathfrak{m}$ (i.e. its number $n$ of
edges) is denoted $|\mathfrak{m}|$.

\begin{figure}[h]
\centerline{\includegraphics[scale=1.0]{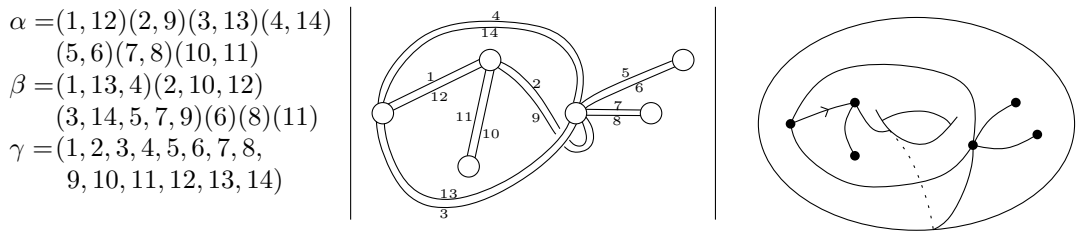}}
\caption{Three different pictures of the same rooted unicellular map.}
\label{fig:fatgraph}
\end{figure}
The definition of a map has a graphical interpretation in terms of labelled fat
graphs (see Figure~\ref{fig:fatgraph}). Roughly speaking, a fat graph is a graph (with
loops and multiple edges allowed), with a prescribed cyclic order of the edges
around each vertex.
Given a map $\mathfrak{m}=(\alpha, \beta, \gamma)$, 
consider the graph $G$
with vertex set the set of cycles of $\beta$, edge set the set of cycles of
$\alpha$, and the natural incidence relation $v \sim e$ if $v$ and $e$ share an
element. 
Now, draw the graph $G$ as follows\footnote{We warn the reader that in another
(an maybe more used) graphical interpretation of maps, edges are cut in their middle and not lengthwise:
 be careful to any confusion.}:
\begin{itemize}
  \item draw edges as ribbons, i.e. such that each edge is split lengthwise
  into two half-edges. These half-edges are labelled by elements of
  $\llbracket 1,2n \rrbracket$
  \item by convention, if a half-edge $i$ belongs to a vertex $v$, then when
  leaving the vertex $v$ by the edge $(i,\alpha(i))$, one sees $\alpha(i)$ on
  the right of $i$.
  \item around each vertex $v$ the half-edges belonging to $v$ read in
  clockwise order are given by the cycle representation of $v$.
\end{itemize}
Observe that the permutation
$\gamma=\beta\alpha$ interprets as follows: start at an half-edge, visit its
associated (opposite) half-edge, and then turn once clockwise around the
vertex.
By repeating this operation, one walks along the half-edges of the graph
without crossing them, so that the cycles of $\gamma$ correspond to the
\emph{borders} of the fat graph. Observe also that a map is connected if and
only if its associated graph is.

The third definition of a map is topological. A map can be defined as a proper
 embedding of a graph (with loops and multiple edges allowed) in a
 compact orientable surface, such that its complementary is a disjoint union of simply
 connected domains (called the faces), and considered up to
 oriented homeomorphism. If furthermore the half-edges are labelled, these
 objects are in bijection with fat graphs (intuitively, one passes
 from a fat graph to a map by gluing a topological polygon along each border,
 hence creating a face; a general account on the equivalence between the three
 definitions of a map can be found in \cite{Mohar-Thomassen}). In particular,
 there is only one orientable surface into which $\mathfrak{m}$ can be properly embedded. 
 If $\mathfrak{m}$ is connected, the genus $g$ of this
 surface is called the \emph{genus} of $\mathfrak{m}$, and recall that we have from Euler
characteristic formula:
$$
|\beta|+|\gamma|=|\alpha|+2-2g
$$

A \emph{unicellular map} is a map which has only one face. Observe that a
unicellular map is necessarily connected. A plane tree is a unicellular map of
genus $0$ (this matches the classical definition, but be careful that this
excludes the tree reduced to a single vertex). Observe
that if a unicellular map has genus $g$, $v$ vertices and $n$ edges, one has: $
n=2g-1+v
$ 
so that the graphs of unicellular maps of positive genus always contain cycles,
and are never trees, in the graph sense.

The \emph{root} of a map is the half-edge numbered $1$. The \emph{root vertex}
(resp. \emph{root edge}) is the vertex (resp. \emph{edge}) containing the root
half-edge. In the topological representation of a map, we represent the root as 
an arrow drawn on the root edge that leaves the root vertex.
A \emph{rooted map} is an equivalence class of maps up to
relabeling of $\llbracket 2,2n\rrbracket$ (i.e. an orbit under the action of
$\mathrm{Stab}(1)$ by conjugation). In the case of unicellular maps, it will
often be convenient to fix a representative: the \emph{canonical representative} of a rooted unicellular map 
is its only representative that satisfies $\gamma = (1,2,\ldots,2n)$.

Finally, we let $\mathcal{U}_{g,n}$ be the set of all rooted unicellular maps of
genus $g$ with $n$ edges, and $\mathcal{U}_g=\cup_n\mathcal{U}_{g,n}$. 
That last notational convention will be used all through the paper: if
$\mathcal{C}_g$ is a class of maps of genus $g$ (where $\mathcal{C}$ may be
replaced by any letter), $\mathcal{C}_{g,n}$ denotes the set of elements of
$\mathcal{C}_g$ with $n$ edges.

\subsection{The slicing and gluing operations}

We now define two operations that we will be essential later.

\subsubsection{slicing a vertex}

Let $\mathfrak{m}=(\alpha,\beta,\gamma)$ be a map, and let 
$v=(i_1,\ldots,i_k)$ be a vertex of $\mathfrak{m}$ of degree $k$.
Let $C$ be a subset of $\{i_1,\ldots,i_k\}$ of cardinality $p$. Up to a cyclic
change in the writing of $v$, we can assume that $i_1\in C$, and write:
$C=\{i_1,i_{l_2},\ldots,i_{l_p}\}$ with $1<l_2<\ldots<l_p \leq k$.
The \emph{slicing} of $v$ by $C$ is the
permutation $\bar{v}$ of $\{i_1,i_2,\ldots,i_k\}$ whose cyclic representation
is : $$
\bar{v}=(i_1, \ldots, i_{l_2-1})
(i_{l_2}, \ldots, i_{l_3-1})
\ldots\ldots
(i_{l_{p}}, \ldots, i_{k})
$$
Let $\bar{\beta}$ be the permutation obtained by replacing $v$ by $\bar v$ in
the cycle representation of $\beta$, and let $\bar \gamma =
\bar\beta\alpha$. We say that the map $\bar m = (\alpha,\bar \beta,
\bar\gamma)$  has been obtained from $\mathfrak{m}$ by the slicing of $v$ by $C$. Observe that the map $\bar{\mathfrak{m}}$ is not necessarily connected, even if
the map $\mathfrak{m}$ is.

The slicing of a vertex is easily interpreted in terms of fat graphs. Given a
vertex $v$ and a set $C$ of $p$ half-edges incident to $v$, replace $v$ by $p$ new
vertices, each incident to one half-edge of $C$. Then, dispatch the other
half-edges by keeping the general cyclic order, and such that the elements of
$C$ have no half-edges on their right. See Figure~\ref{fig:slicing-gluing}.
\begin{figure}[h]
\centerline{\includegraphics[scale=1.0]{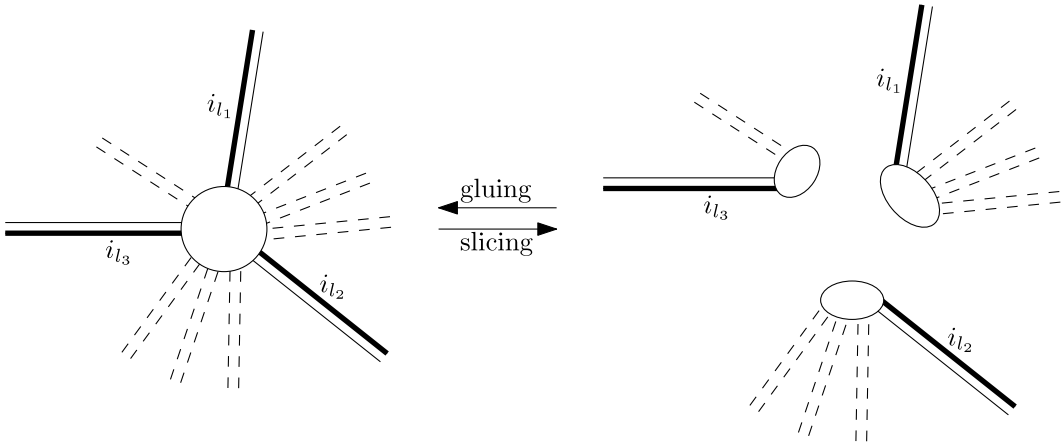}}
\caption{The gluing and slicing operations for $3$ half-edges.}
\label{fig:slicing-gluing}
\end{figure}

\subsubsection{gluing half-edges}

Let $\mathfrak{m}=(\alpha,\beta,\gamma)$ be a map, and let 
$c=(i_1,i_2,\ldots,i_k)$ be an \emph{ordered}  $k$-tuple of half-edges incident to
different vertices. 
Then, each $i_l$ is incident to a vertex $v_l$, of cycle
representation: $v_l = (i_l,j^{l}_{1}, j^{l}_{2}, \ldots j^{l}_{n_l})$, with
$n_l \geq 0$. 
The \emph{gluing} of $v_1,v_2,\ldots,v_k$ by $c$ is the cyclic
permutation $\bar{v}$ defined by: 
$$
\bar v = (i_1,j^{1}_{1}, \ldots,j^{1}_{n_1}, 
i_2,j^{2}_{1}, \ldots,j^{2}_{n_2},
\ldots, 
i_k,j^{k}_{1}, \ldots,j^{k}_{n_k} 
)
$$
Let $\bar \beta$ be the permutation obtained from the cycle representation of
$\beta$ by erasing the cycles $v_1,\ldots,v_k$ and replacing them with
$\bar{v}$, and let $\bar\gamma=\bar\beta\alpha$. We say that the map
$\bar{\mathfrak{m}}=(\alpha,\bar\beta,\bar\gamma)$ has been obtained from
$\mathfrak{m}$ by gluing $c$. The new map $\bar{\mathfrak{m}}$ has 
$|\beta|-k+1$ vertices, and the same number of edges.

Observe that the gluing and slicing operations are reciprocal one to another
(see Figure~\ref{fig:slicing-gluing} again).

\section{Schemes and dominant maps.}

\label{sec:scheme}

In this section, we describe a technique already written
in~\cite{ChMaSc} that enables to perform the asymptotic enumeration of
unicellular maps via generating series techniques. It consists in the reduction
of unicellular maps to a finite number of objects, called schemes. We need that
in order to identify the dominating case, and define precisely what a
\emph{dominant map} is.

\subsection{The scheme of a unicellular map}
\label{subsec:scheme}

\begin{definition}
A \emph{scheme} of genus $g$ is a rooted unicellular map of genus $g$ without
vertices of degree $1$ nor $2$. $\mathcal{S}_g$ is the set of all schemes of
genus $g$.
\end{definition}
If a scheme of genus $g$ has $n_i$ vertices of degree $i$ for all $i$, the
fact that $n_1=n_2=0$ and Euler characteristic formula implies:
\begin{eqnarray}
\label{eq:ni}
\sum_{i\geq 3}\frac{i-2}{2}n_i = 2g-1
\end{eqnarray}
In particular, the sequence $(n_i)_{i\geq 1}$ can only take a finite number of
values, which implies the following lemma:
\begin{lemma}[\cite{ChMaSc}]
For every $g\geq 1$, the set $\mathcal{S}_g$ of schemes of genus $g$ is finite.
\end{lemma}
It will be convenient (and possible) to assume that each scheme carries a
fixed labelling and orientation of the edges: i.e., we shall
speak of the $i$-th edge of a scheme, and of its orientation without
more precision. By convention, the $1$-st edge of a scheme will be the root edge.

\begin{figure}[h]
\centerline{\includegraphics[scale=1.0]{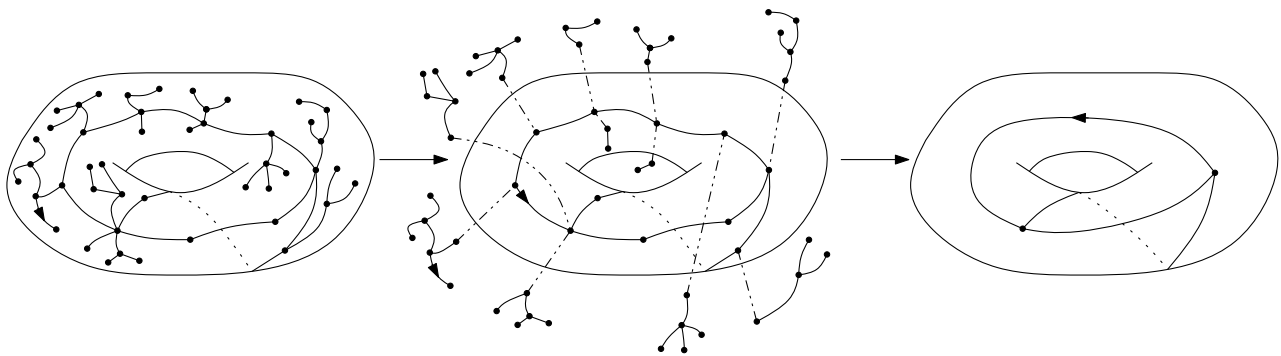}}
\caption{From a unicellular map of genus $1$ to its scheme (observe the
rooting of the core).}
\label{fig:scheme}
\end{figure}

\begin{figure}
\centerline{\includegraphics[scale=1.0]{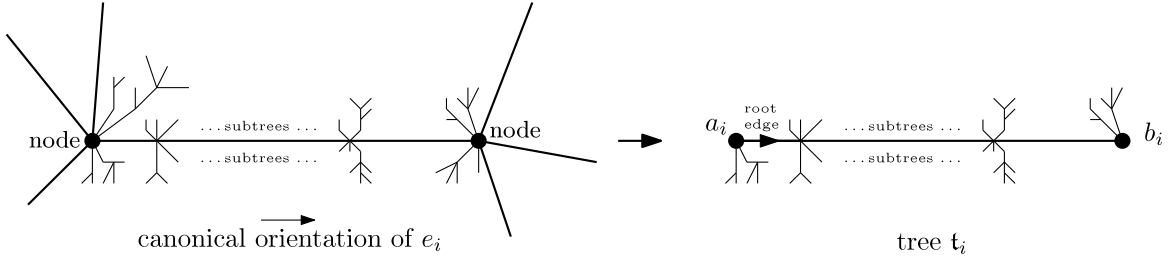}}
\caption{The tree $\mathfrak{t}_i$ associated to the $i$-th edge of the scheme.}
\label{fig:supertree}
\end{figure}

We now explain how to associate a scheme to a unicellular map  (see
Figure~\ref{fig:scheme}). Let $\mathfrak{m}$ be a rooted unicellular map of
genus $g$. First, we erase recursively all the vertices of $\mathfrak{m}$ of degree $1$, until 
there are no more such vertices left. We are left with a  map $\mathfrak{c}$,
witout vertices of degree $1$, which we call the \emph{core} of $\mathfrak{m}$. 
By convention, $\mathfrak{c}$ is rooted as follows: if the root edge of
$\mathfrak{m}$ is still present in $\mathfrak{c}$, we keep it as the root of
$\mathfrak{c}$. Otherwise, the root belongs to some plane tree which is
attached to some vertex $v$ of $\mathfrak{c}$: the root edge of $\mathfrak{c}$
is the first edge of $\mathfrak{c}$ encountered after that plane
tree when turning counterclockwise around $v$ (and it is
oriented leaving $v$). Now, in the core, 
the vertices of degree $2$ are organised into maximal chains connected together
at vertices of degree at least $3$. We replace each of these chains by a new
edge: we obtain a map $\mathfrak{s}$ without vertices of degree $1$ nor $2$. 
The root of $\mathfrak{s}$ is the new edge corresponding to the chain that was
carrying the root of $\mathfrak{c}$ (with the same orientation).
By construction, $\mathfrak{s}$ is a scheme of genus $g$, called \emph{the
scheme of} $\mathfrak{m}$. The vertices of $\mathfrak{m}$ that remain vertices
in its scheme are called the \emph{nodes} of $\mathfrak{m}$.

Now, say that $\mathfrak{s}$ has $k$ edges and $l$ vertices.
Let $v_1,\ldots,v_l$, be the nodes of $\mathfrak{m}$, and assume that each $v_i$
is incident to exaclty $n_i$ edges of the core of $\mathfrak{m}$, say
$h^{(i)}_1,\ldots h^{(i)}_{n_i}$. Let $\mathfrak{t}_*$ be the map obtained by
splitting $\mathfrak{m}$ successively at all the
$(h^{(i)}_1,\ldots h^{(i)}_{n_i})$. It is
easily seen that $\mathfrak{t}_*$ has $k$ connected components, 
each of them being a plane tree. Let $\mathfrak{t}_i$
be the connected component associated to the $i$-th edge (say $e_i$) of
$\mathfrak{s}$. Let $a_i$ and $b_i$ be the vertices of $\mathfrak{t}_i$  
corresponding respectively to the origin and endpoint of $e_i$ in
$\mathfrak{s}$ (recall that $e_i$ is canonically oriented). In $\mathfrak{t}_i$,
there is a unique simple path from $a_i$ to $b_i$: by convention, we let the
first (oriented) edge of this path be the root of $\mathfrak{t_i}$
(see Figure~\ref{fig:supertree}).
\begin{definition}
We let $\mathfrak{T}$ be the set of pairs $(\mathfrak{t},\nu)$, where
$\mathfrak{t}$ is a rooted plane tree, and $\nu$ is a vertex of $\mathfrak{t}$
different from the root vertex, such that the unique simple path in
$\mathfrak{t}$ that goes from the root vertex to $\nu$ contains the root edge.
\end{definition} 
Now, let $(\mathfrak{t},\nu)\in\mathfrak{T}$ and $\epsilon$ be an oriented edge
of $\mathfrak{t}$. If $\epsilon$ is not an edge of the oriented path $p(\nu)$ from
the root vertex to $\nu$ in $\mathfrak{t}$, it belongs to some tree $\tau$ attached at a
vertex $v$ of $p(\nu)$. In this case, we say that $\epsilon$ is \emph{at the
right of $\nu$} if either $v$ is the root vertex, either $v\neq \nu$ and the tree $\tau$ is
attached at the right of $p(\nu)$. In the other case, i.e. if $\epsilon\in
p(\nu)$, we say that it is at the right of $\nu$ if it is oriented as $p(\nu)$.
We let $\widetilde{\mathfrak{T}}$ be the set of triples
$(\mathfrak{t},\nu,\epsilon)$, where $(\mathfrak{t},\nu)\in\mathfrak{T}$ and
$\epsilon$ is at the right of $\nu$.

Observe that, from the construction above, each
$(\mathfrak{t}_i,b_i)$ is an element of $\mathfrak{T}$. 
Moreover, the root
edge of $\mathfrak{m}$ makes $(\mathfrak{t}_1,b_1)$ an element of
$\widetilde{\mathfrak{T}}$,
so that we have associated to the map $\mathfrak{m}$ an element of 
$\widetilde{\mathfrak{T}}\times\mathfrak{T}^{k-1}$.

Conversely, given a scheme $\mathfrak{s}$ with $k$ edges and an element
$\mathfrak{t}_*\in\widetilde{\mathfrak{T}}\times\mathfrak{T}^{k-1}$, one easily
reconstruct a map $m$ of scheme $\mathfrak{s}$ by replacing the $i$-th edge of $\mathfrak{s}$
by the $i$-th tree $\mathfrak{t}_i$ as in Figure~\ref{fig:supertree}.
This construction is clearly reciprocal to the previous one, which gives:
\begin{lemma}
Let 
$ \displaystyle
T(z) = \sum_{\mathfrak{t}\in\mathfrak{T}} z^{|\mathfrak{t}|}, $ 
$ \displaystyle
\widetilde{T}(z) = \sum_{\mathfrak{t}\in\widetilde{\mathfrak{T}}}
z^{|\mathfrak{t}|} $, and let $\displaystyle
U_g(z) = \sum_{\mathfrak{m} \in \mathcal{U}_{g}} z^{|\mathfrak{m}|}
$
be the generating series of unicellular maps of genus $g$ by the number of
edges. Then one has:
\begin{eqnarray}
\label{eq:allschemes}
U_g(z) = \sum_{\mathfrak{s}\in\mathcal{S}_g} \widetilde{T}(z)T(z)^{|\mathfrak{s}|-1}
\end{eqnarray}
\end{lemma}

It is easy to
compute the series $T(z)$ and $\widetilde{T}(z)$. Let
$\mathfrak{T}'$ be the set of rooted trees with a marked vertex. 
It is clear that the operation defined
by the flipping of the root edge is an involution of $\mathfrak{T}'$ that sends $\mathfrak{T}$
to its complementary
(recall
that our trees have at least one edge). 
Hence the number of elements of $\mathfrak{T}$ with $n$
edges is: $$
\left|\mathfrak{T}_n\right| = \frac{1}{2}\left|\mathfrak{T}'_n\right|
=\frac{n+1}{2}\frac{1}{n+1}{2n \choose n},
$$
which gives:
$$T(z)= \frac{1}{2}\left(\frac{1}{\sqrt{1-4z}}-1\right)
$$
Now, let $\mathfrak{T}''$ be the set of elements
$(\mathfrak{t},\nu)\in\mathfrak{T}$ that carry an additional  distinguished
oriented edge. Inverting the roles of the root edge and of the opposite of the last edge of the
path $p(\nu)$ is an involution of $\mathfrak{T''}$ that sends
$\widetilde{\mathfrak{T}}$ to its complementary. Hence:
$$
\widetilde{T}(z) = \frac{1}{2}\sum_n 2n \left|\mathfrak{T}_n\right|z^n =
\frac{zd}{dz} T(z)$$
This last equation and (\ref{eq:allschemes}) imply:
\begin{eqnarray}
\label{eq:doublerooting}
U_g(z) =\frac{zd}{dz}
\sum_{\mathfrak{s}\in\mathcal{S}_g}
\frac{1}{|\mathfrak{s}|}T(z)^{|\mathfrak{s}|}
\end{eqnarray}

\subsection{The double-rooting argument}
\label{subsec:doublerooting}

We now explain how to prove this last equation directly, without
computing $\widetilde{T}$. Recall that $\mathcal{S}_{g,k}$ is the set of rooted
schemes of genus $g$ with $k$ edges. Given a scheme
$\mathfrak{s}\in\mathcal{S}_{g,k}$ and an element of 
$\mathfrak{T}^k$, substituting each tree with the
corresponding edge of $\mathfrak{s}$, and then selecting an oriented root
edge in the obtained map, one builds a rooted map, whose scheme is equal to
$\mathfrak{s}$ as an unrooted map, but may have a different rooting. In other
words, that scheme carries an additionnal distinguished oriented edge (given by
the root of $\mathfrak{s}$). Such a map
can also be constructed by starting first with a rooted map whose scheme has
$k$ edges, and then selecting an oriented edge of its scheme. 
Hence, if
$U_{\mathfrak{s}}(z)$ denotes the series of rooted maps of scheme $\mathfrak{s}$, we have:
$$
\sum_{\mathfrak{s}\in\mathcal{S}_{g,k}}2\cdot\frac{zd}{dz}
T(z)^{k} = 
\sum_{\mathfrak{s}\in\mathcal{S}_{g,k}}2k\cdot U_{\mathfrak{s}}(z)
$$ so that:
$$
\sum_{\mathfrak{s}\in\mathcal{S}_{g,k}}U_{\mathfrak{s}}(z)=\frac{1}{k}
\sum_{\mathfrak{s}\in\mathcal{S}_{g,k}}\frac{zd}{dz}
T(z)^{k} 
$$
which gives another proof
of Equation~\ref{eq:doublerooting} after summing on $k$. We call this trick the
``double-rooting argument''
(this argument already
appears in \cite{ChMaSc}).
Observe that to enumerate rooted unicellular
maps, we enumerate unicellular maps which are ``doubly-rooted'' (i.e. they have at the same time a
root and a distinguished oriented edge of their scheme), but we count them
with a weight inverse of the size of their scheme. Observe also that we could
have been more direct and avoid the summation on $k$ (proving directly that
$U_\mathfrak{s}(z)=\frac{1}{|\mathfrak{s}|}\frac{zd}{dz}
T(z)^{|\mathfrak{s}|} $). We chose purposely this presentation because we will
use later a variant of this argument, adapted to the case of labelled trees, 
where that approach will be necessary.

\subsection{The dominant schemes}

For all $k\geq 1$, the generating series
$\frac{zd}{dz}T(z)^k$ 
satifies the following expansion near its radius of convergence $1/4$:
$$
\frac{zd}{dz}T(z)^{k} \sim
A_k(1-4z)^{-\frac{k}{2}-1} \:\:, \:\:
z \rightarrow
1/4,$$ 
for some positive constant $A_k$.
Moreover, as an algebraic series, it is amenable to singularity analysis, in
the sense of \cite{FlOd}: it follows from the transfer theorems of \cite{FlOd} that the $n$-th coefficient in the series expansion of 
$\frac{zd}{dz}T(z)^{k}$ is equivalent to $B_kn^{k/2} 4^n$ for some constant $B_k>0$.

Now, in the sum~(\ref{eq:doublerooting}), the maximum value of $k$ 
is realized by schemes with the maximal number of
edges. Moreover, Equation~\ref{eq:ni} and Euler characteristic formula imply that
those schemes of genus $g$ that have the maximal number of edges are the ones which 
have only vertices of degree $3$. Such a scheme has $6g-3$ edges and $4g-2$
vertices.
This leads to the following definition:
\begin{definition}
A \emph{dominant scheme} of genus $g$ is a rooted unicellular map of genus $g$
whith $4g-2$ vertices of degree $3$ and $6g-3$ edges.
A \emph{dominant unicellular map} is a unicellular map whose scheme is dominant.
$\mathcal{S}_g^*$ (resp. $\mathcal{U}_{g}^*$) is the set
of all dominant schemes (resp. dominant maps) of
genus $g$.
\end{definition}
From the discussion above, and after regrouping in Equation~\ref{eq:doublerooting} the terms corresponding to schemes of the same size, 
we see that the number of elements in $\mathcal{U}_{g,n}^*$ is equivalent to $B n^{(6g-3)/2} 4^n$, for some $B>0$,
whereas the number of elements in $\mathcal{U}_{g,n}\setminus\mathcal{U}_{g,n}^*$ is at most $B' n^{(6g-4)/2} 4^n$, for some other constant $B'>0$.
This gives:
\begin{lemma}
\label{lemma:dominant}
Fix $g\geq 1$. When $n$ tends to infinity, one has:
$$
\frac{\left|\mathcal{U}^*_{g,n}\right|}{\left|\mathcal{U}_{g,n}\right|}
=1-O\left(\frac{1}{\sqrt{n}}\right)$$
\end{lemma}

\section{The number of intertwined nodes: the key lemma.}

\label{sec:key}

\subsection{Notion of intertwined node}
Let $\mathfrak{m}=(\alpha,\beta,\gamma)$ be a dominant unicellular map of genus
$g\geq 1$. We assume that $\mathfrak{m}$ is given in its canonical labeling,
i.e. that $\gamma=(1,2,\ldots,2n)$.
Let $v$ be a node of $\mathfrak{m}$. $v$ is incident to exactly three half-edges that
belong to the core of $\mathfrak{m}$. Let $e_1,e_2,e_3$ be the labels 
of these half-edges {\bf read in clockwise order} around $v$.
Since $e_1,e_2,e_3$ are defined up to a cyclic permutation, we can assume that
$e_1$ is the minimum: $e_1<e_2$ and $e_1<e_3$.
The following definition is short but fondamental:
\begin{definition}
We say that $v$ is an \emph{intertwined node} iff $e_1<e_3<e_2$.
\end{definition}

The definition is motivated by the following lemma:
\begin{lemma}
\label{lemma:intertwined}
If $v$ is an intertwined node, then the map $\bar{\mathfrak{m}}$ obtained after
slicing $v$ by $\{e_1,e_2,e_3\}$ is connected, and it is a dominant unicellular
map of genus $g-1$. We note: $\bar{\mathfrak{m}}=\mathfrak{m}\setminus v$.
\end{lemma}
\begin{proof}
Let us write $v=(e_1,i^1_1,\ldots,i^1_{n_1},
e_2,i^2_1,\ldots,i^2_{n_2},e_3,i^3_1,\ldots,i^3_{n_3})$.
Since $\gamma=(1,2,\ldots,2n)$, when turning counterclockwise around the unique
face of $\mathfrak{m}$, starting at $e_1$, one sees a certain number of half-edges,
and then $e_3$ before $e_2$. Then, the last edge seen before $e_3$ is
$\gamma^{-1}(e_3)$ = $\alpha(i_{n_2}^2)$. Going on along the face, one sees then
$\alpha(i_{n_1}^1)$, $e_2$, $\alpha(i_{n_3}^3)$, and $e_1$ again, in that order.
In other words, we can write the graph of $\gamma$ as follows:
$$
\gamma: \; 
e_1 \rightarrow \stackrel{1}{\ldots} \rightarrow \alpha(i^2_{n_2}) 
    \rightarrow
e_3 \rightarrow \stackrel{3}{\ldots} \rightarrow \alpha(i^1_{n_1}) 
	\rightarrow
e_2 \rightarrow \stackrel{2}{\ldots} \rightarrow \alpha(i^3_{n_3}) 
	\rightarrow 
e_1$$
where $i\rightarrow j $ means that $\gamma(i)=j$, and where the notation
$\stackrel{i}{\ldots}$ denotes a sequence of the form $j^i_1\rightarrow j^i_2
\rightarrow \ldots \rightarrow j^i_{m_i}$.

Now, let $\bar\beta$ be the permutation obtained after slicing $v$ by
$\{e_1,e_2,e_3\}$, and let $\bar \gamma=\bar\beta\alpha$. By definition,
$\bar \beta$ is obtained from $\beta$ by replacing $v$ by:
 $$\bar v = (e_1,i^1_1,\ldots,i^1_{n_1})
(e_2,i^2_1,\ldots,i^2_{n_2})(e_3,i^3_1,\ldots,i^3_{n_3})
$$
Hence the only arrows to modify in the graph of $\gamma$ to obtain the
graph of $\bar \gamma$ are the ones leaving
$\alpha(i^1_{n_1})$, $\alpha(i^2_{n_2})$, $\alpha(i^3_{n_3})$.
Now, one has $\bar\gamma(\alpha(i^1_{n_1}))=\bar\beta(i^1_{n_1})=e_1$. In the
same way one has: $\bar\gamma(\alpha(i^2_{n_2}))=e_2$ and 
$\bar\gamma(\alpha(i^3_{n_3}))=e_3$. Thus the graph of $\bar\gamma$ is:
$$\bar\gamma: 
e_1 \rightarrow \stackrel{1}{\ldots} \rightarrow \alpha(i^2_{n_2}) 
    \rightarrow
e_2 \rightarrow \stackrel{2}{\ldots} \rightarrow \alpha(i^3_{n_3}) 
	\rightarrow
e_3 \rightarrow \stackrel{3}{\ldots} \rightarrow \alpha(i^1_{n_1}) 
	\rightarrow 
e_1$$
for the same dotted sequences.
Hence, the
cycle of $\bar \gamma$ containing $e_1$ has length $2n$, so that it is the only cycle of $\bar\gamma$. This proves at the
same time that 
$\bar{\mathfrak{m}}$ is connected, and that it is unicellular. Moreover,
$\bar{\mathfrak{m}}$ has the same number of edges as $\mathfrak{m}$, 
two more vertices, and both have one face, so that Euler characteristic
formula implies that it has genus $g-1$.

Finally, let us construct the core of $\bar{\mathfrak{m}}$, via the algorithm
of Section 2.1. It is clear that all the edges that are
erased during the construction of the core of $\mathfrak{m}$ are erased during
the construction of the core of $\bar{\mathfrak{m}}$, so that all the edges of
the core of $\bar{\mathfrak{m}}$ are edges of the core of $\mathfrak{m}$.
Consequently, the vertices of the scheme of $\bar{\mathfrak{m}}$ cannot have
degree more than $3$, i.e. $\bar{\mathfrak{m}}$ is dominant.
\end{proof}

\subsection{The key lemma}
\begin{lemma}
\label{lemma:key}
Let $\mathfrak{m}\in \mathcal{U}^*_g$ be a dominant map of genus $g$. Then
$\mathfrak{m}$ has exactly $2g$ intertwined nodes.
\end{lemma}
\begin{proof}
\begin{figure}[h]
\centerline{\includegraphics[scale=1]{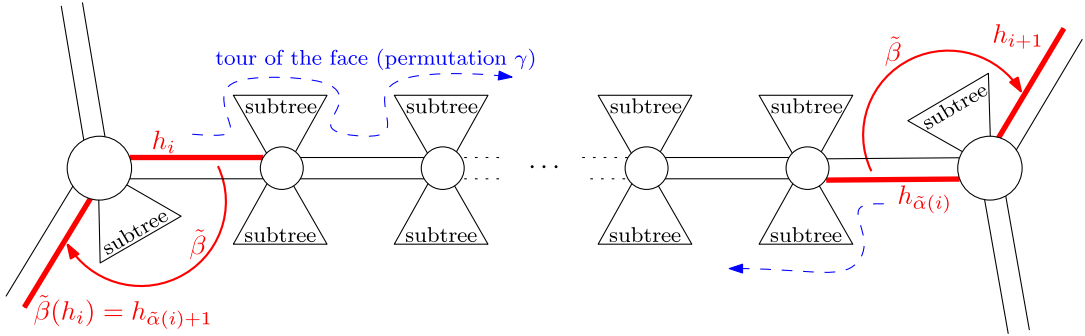}}
\caption{The permutation $\tilde \alpha$.}
\label{fig:doublecounting}
\end{figure}
We assume that $\mathfrak{m}=(\alpha,\beta,\gamma)$ is given in canonical
form: $\gamma=(1,2,\ldots,2n)$. 
Let $k=4g-2$, and let $h_1 < \ldots < h_{3k}$ be the
(labels of) half-edges of the core of $\mathfrak{m}$ which are incident to a node. $\tilde\beta$ is the
restriction of the permutation $\beta$ to the $h_i$'s, i.e.:
$
\tilde\beta(h_i)=\beta^{k_i}(h_i) \mbox{ where } k_i = \mathrm{min} \{k\geq 1:
\beta^k(h_i)\in \{h_1,\ldots h_{3k}\}\} $
We say that a half-edge $h_i$ is \emph{increasing} if $h_i<
\tilde\beta(h_i)$, and \emph{decreasing} otherwise.
We let $n_+$ (resp. $n_-$) be the number of increasing (resp. decreasing)
half-edges. Observe that $n_++n_-=3k$.

Let us fix an half-edge $h_i$. There is an unique
$\tilde \alpha(i)\in\llbracket 1,3k \rrbracket$ such that
$\tilde\beta(h_i)=h_{\tilde\alpha(i)+1}$ (with the convention $3k+1=1$). Now,
it is clear from the fact that $\gamma=(1,2,\ldots,2n)$ and Figure~\ref{fig:doublecounting} that 
$\tilde\beta (h_{\tilde\alpha(i)}) = h_{i+1}$ (with the same convention).
In other words, the application $\tilde\alpha$ is an involution  without fixed points of
$\llbracket 1,3k \rrbracket$. Moreover, observe that if $3k\not\in
\{i,\tilde\alpha(i)\}$ then one has either
$i<i+1<\tilde\alpha(i)<\tilde\alpha(i)+1$, either
$\tilde\alpha(i)<\tilde\alpha(i)+1<i<i+1$, so that
$h_i$ is increasing if and only if
$h_{\tilde\alpha(i)}$ is decreasing. Else, say if $i=3k$, then
$\tilde\beta(h_{\tilde\alpha(i)})=h_1$, so that both $h_{3k}$ and 
$h_{\tilde\alpha(3k)}$ are decreasing, since $h_1$ and $h_{3k}$ are
respectively the smallest and largest of the $h_i$'s. Consequently there are two
more decreasing than increasing half-edges: $n_-=n_++2$, which gives $n_+=\frac{3}{2}k-1=6g-4$.

Finally, by definition, an intertwined node has exactly one increasing half-edge, 
whereas a non-intertwined node has exactly two of them.
Hence, if $\iota$ is the number of intertwined nodes, one has:
$$
n_+ = \iota + 2 (4g-2-\iota)
$$
which gives: $\iota=8g-4-n_+=2g$
\end{proof}

\section{Opening sequences and the bijection.}

\label{sec:opening}

\subsection{Opening sequences}

\begin{definition}
Let $\mathfrak{m}\in \mathcal{U}^*_g$ be a dominant unicellular map of genus
$g$. An \emph{opening sequence} of $\mathfrak{m}$ is a $g$-uple
$v_*=(v_1,\ldots v_g)$ such that:
\begin{itemize}
  \item $v_g$ is an intertwined node of $\mathfrak{m}$ 
  \item for all $i\in\llbracket 1,g-1\rrbracket$, $v_{i}$ is an intertwined node
  of $\mathfrak{m}\setminus v_g\setminus\ldots\setminus v_{i+1}.$
\end{itemize}
An \emph{opened map} is a dominant unicellular map together with an opening
sequence. The set of all opened maps of genus $g$ is denoted $\mathcal{O}_g^*$.
\end{definition}

From Lemma~\ref{lemma:key}, each dominant unicellular map
$\mathfrak{m}$ of genus $g$ has exactly $2g$ intertwined nodes. Now, once such
a node $v_g$ has been chosen, the map $\mathfrak{m}\setminus v_g$ is, from 
Lemma~\ref{lemma:intertwined}, a dominant unicellular map of genus $g-1$. So,
using Lemma~\ref{lemma:key} again, it has itself $2(g-1)$ intertwined nodes.
Repeating the argument $g$ times, one obtains the following proposition:
\begin{proposition}
\label{prop:2^gg!}
Let $\mathfrak{m}\in\mathcal{U}_g^*$. Then $\mathfrak{m}$ has exactly
$\displaystyle
\prod_{i=1}^{g} (2i) = 2^g g!
$
opening sequences.
The numbers of opened maps and dominant unicellular maps are related by:
$\displaystyle
\left|\mathcal{O}^*_{g,n}\right|=2^g g!
\left|\mathcal{U}^*_{g,n}\right| $
\end{proposition}

Note: in what follows, to shorten notation, we note $\mathfrak{m}\setminus
v_g\ldots v_i$ for $\mathfrak{m}\setminus v_g\setminus\ldots\setminus v_{i}$.

\subsection{Trees with $g$ triples}
\label{subsection:skeleton}

Let $\mathfrak{t}$ be a tree of vertex set $V$ and edge set $E$. Let
$W\subset V$ be a subset of the vertices of $\mathfrak{t}$.
For $v,v'\in V$, let $p(v,v')$ be the subset of $E$ made of all the edges of
the unique simple path going from $v$ to $v'$ in $\mathfrak{t}$.
We set:
$$\mathfrak{r}(W) = \bigcup_{(v,v')\in W^2} p(v,v').$$
If the root edge of $\mathfrak{t}$ is still present in $\mathfrak{r}(W)$, we
keep it as the root of $\mathfrak{r}(W)$. Otherwise, it belongs to some subtree 
which is attached to some vertex $v$ of $\mathfrak{r}(W)$: we select the first
edge of $\mathfrak{r}(W)$ encountered counterclockwise around $v$ after that
tree as the root edge of $\mathfrak{r}(W)$, and orient it leaving $v$.

In $\mathfrak{r}(W)$, the vertices of degree $2$ which do not belong to $W$
are organised into maximal chains, whose extremities are either
elements of $W$, either vertices of degree $\geq 3$. We now replace
each of these maximal chains by a new edge: we obtain a tree
$\mathfrak{s}(W)$, which inherits naturally a root from the root of
$\mathfrak{r}(W)$. We say that $\mathfrak{s}(W)$ is the
\emph{skeleton} of $W$ in $\mathfrak{t}$.
\begin{figure}[h]
\centerline{\includegraphics[scale=1.0]{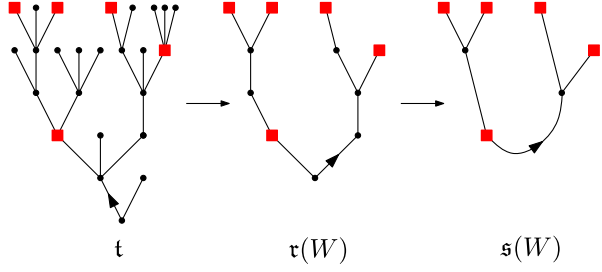}}
\caption{The skeleton of a rooted plane tree with 4 marked vertices (elements
of $W$ are squares).}
\label{fig:skeleton}
\end{figure}

We let $\mathcal{X}^p_k$ be the set of all possible skeletons with $k$ edges of
a set of $p$ elements: $$
\mathcal{X}^p_{k} = \{\mathfrak{s} \mbox{ rooted tree with } k \mbox{ edges}, \:
\exists
\mathfrak{t}
\mbox{ and }W,\: |W|=p \mbox{ and } \mathfrak{s}\mbox{ is the skeleton of }
W \mbox{ in }\mathfrak{t}\}
$$
Observe that, since all the vertices of $\mathfrak{s}(W)$ of degree $1$ and $2$
are elements of $W$, if $n_i$ denotes the number of vertices of degree $i$ in
$\mathfrak{s}(W)$, we have from Euler characteristic formula:
$\displaystyle
\sum_{i\geq 3}\frac{i-2}{2}n_i = \frac{n_1}{2}-1 \leq \frac{|W|}{2}-1
$
and $n_2\leq |W|$, so that for a fixed value of $|W|$, the number of possible
skeletons is finite.

\begin{definition}
Let $\mathfrak{t}$ be a tree and $W$ be a subset of the vertices of
$\mathfrak{t}$. We say that $W$ is \emph{non-singular} if its skeleton $\mathfrak{s}(W)$ has
only vertices of degree $1$ and $3$, and if all the elements of $W$ have
degree $1$ in $\mathfrak{s}(W)$.
\end{definition}

\begin{definition}[Trees with $g$ triples]
A \emph{tree with $g$ triples} is a pair $(\mathfrak{t},c_*)$,
where
\begin{itemize}
  \item $\mathfrak{t}$ is a rooted plane tree,
  \item $c_*=(c_1,\ldots, c_g)$ where each $c_i$ is a subset of the
  vertices of $\mathfrak{t}$ with three elements:
  $c_i=\{v^{(i)}_1,v^{(i)}_2,v^{(i)}_3\}$.
  \item the $c_i$ are disjoint: $i\neq j \Rightarrow c_i \cap c_j = \emptyset$
  \item the set $\displaystyle\bigcup_i c_i$ is non-singular. 
\end{itemize}
The set of all trees with $g$ triples is denoted
$\mathcal{T}_{g}$.
\end{definition}

Let $(\mathfrak{t},c_*)$ be a tree with $g$ triples. By abuse of notation, we
also note $c_*$ for $\cup_i c_i$. Fix $v\in c_*$. For all
$x\in c_*\setminus\{v\}$, there is a unique simple path going from $v$ to $x$:
let $e_{v,x}$ be the first edge of this path. If there existed $x$ and $y$ such that
$e_{v,x}\neq e_{v,y}$, then the path linking $x$ and $y$ would visit $v$,
and $v$ would be an inner node of the skeleton, which contradicts the fact that
$c_*$ is non-singular. Hence the edge $e_{v,x}$ depends only on $v$. This edge is made 
of two half-edges, one of them belonging to $v$ (in the sense of
the encoding of maps by permutations). We note this half-edge $h_v$ and say
it is the \emph{incoming half-edge} at $v$.

\begin{proposition}
\label{prop:opening}
Let $(\mathfrak{m},v_*)\in \mathcal{O}^*_{g,n}$ be an opened map of genus $g$
with $n$ edges. 
For $i\in\llbracket 1,g\rrbracket$, the vertex $v_i$
of $\mathfrak{m}\setminus v_g \ldots v_{i+1}$ gives birth to
three vertices, say $v^{i}_1,v^{i}_2,v^{i}_3$, of $\mathfrak{m}\setminus
v_g \ldots v_{i}$.

Let $\mathfrak{t}=\mathfrak{m}\setminus v_g \ldots v_{1}$.
For all $i$, let $c_i=\{v^{i}_1,v^{i}_2,v^{i}_3\}$, and let
$c_*=(c_1,\ldots,c_g)$. Then $(\mathfrak{t},c_*)$ is a tree with $g$ triples.
Moreover, the half-edges of the core of $\mathfrak{m}\setminus v_g\ldots
v_{i+1}$ incident to $v_i$ are the incoming half-edges of
$v^{i}_1,v^{i}_2,v^{i}_3$ in $(\mathfrak{t},c_*)$.
\end{proposition}
\begin{definition}
We note $\Phi(\mathfrak{m},v_*)=(\mathfrak{t},c_*)$. 
The map $\Phi$ is therefore an application: 
$$\Phi: \mathcal{O}^*_{g,n} \longrightarrow \mathcal{T}_{g,n}$$
\end{definition}
\begin{proof}[Proof of the proposition]
It is clear by using Lemma~\ref{lemma:intertwined} recursively that
$\mathfrak{t}$ is a rooted unicellular planar map, i.e. a rooted plane tree.

Fix $i\in\llbracket 1,g \rrbracket$.
In the map $\mathfrak{m}\setminus v_g\ldots v_{i+1}$,
three half-edges of the core, say $h_i^1,h_i^2,h_i^3$, meet at the vertex
$v_i$. Each one gives birth to a vertex $v_i^j$ of $\mathfrak{m}\setminus
v_g\ldots v_{i}$. In this map, if we (temporarily) disconnect the edge
containing $h_i^j$ from $v_i^j$, the connected component containing $v_i^j$ is a tree: indeed, if
it was not a tree, this would imply that $v_i$ is connected to an additional
edge of the core, which contradicts the fact that $\mathfrak{m}$ is dominant. For the same
reason, this tree cannot contain any of the vertices $v_{i'}^{j'}$ for
$(i',j')\neq (i,j)$. Hence in $\mathfrak{m}\setminus v_g\ldots v_{i}$, a path
connecting $v_i^j$ to some $v_{i'}^{j'}$  necessarily begins with the edge containing
$h_i^j$. Since a path in the tree $\mathfrak{t}$ is also a path in the map
$\mathfrak{m}\setminus v_g\ldots v_i$, this last property is still true in
$\mathfrak{t}$. 
On the one hand, this implies all the $v_i^j$ have degree $1$ in the skeleton of
$\cup_i c_i$. Since the other vertices of the skeleton  are exactly those nodes 
of $\mathfrak{m}$ which have not been opened
during the construction of $\mathfrak{t}$, they all have degree $3$:
hence $c_*$ is non singular and $(\mathfrak{t},c_*)$ is a valid tree with $g$ triples.
On the other hand, this implies that $h_i^j$ is the incoming half-edge of
$v_i^j$ in $(\mathfrak{t},c_*)$, which ends the proof of the proposition.
\end{proof}

It is now easy to define the converse application of $\Phi$. We begin with a
lemma:
\begin{lemma}
\label{lemma:onlyonegluing}
Let $\mathfrak{m}$ be a unicellular map of genus $g$ given in its canonical
representation, and let $h_1<h_2<h_3$ be three half-edges of $\mathfrak{m}$
incident to three different vertices. Then the gluing of $\mathfrak{m}$ by 
$(h_1,h_2,h_3)$ creates a unicellular map
$\bar{\mathfrak{m}}$ of genus $g+1$. Moreover, if $\bar{\mathfrak{m}}$ is
dominant, the vertex created by that gluing  is an
intertwined node of $\bar{\mathfrak{m}}$. 

On the
contrary, the gluing of $\mathfrak{m}$ by $(h_1,h_3,h_2)$ creates a map of
genus $g$ with three faces.
\end{lemma}
\begin{proof}
The proof is almost the same as the one of Lemma~\ref{lemma:intertwined}, taken
at reverse. It is easy to check that the
map obtained by gluing $(h_1,h_3,h_2)$ has genus $g$ and three
faces: intuitively, one can draw this gluing inside the unique face of
$\mathfrak{m}$, splitting it into three faces. On the contrary, 
after gluing $(h_1,h_2,h_3)$,
it is easily checked that the obtained map has one face, and that the order 
of appearance of the half-edges around the new vertex during the tour of the face 
matches the definition of an intertwined
node. We leave the details to the reader (one should carefully follow the
tour of the face, as in the proof of Lemma~\ref{lemma:intertwined}).

Finally, Euler characteristic formula implies that $\bar{\mathfrak{m}}$ has
genus $g+1$.
\end{proof}
In particular, the above lemma says that given three half-edges incident to
different vertices in a unicellular map of genus $g$, there is only one gluing
of these half-edges (among the two possible circular permutations) that creates 
a unicellular map (and it has genus $g+1$).

\begin{definition}
Let $(\mathfrak{t},c_*)$ be a tree with $g$ triples. For all $i$ let
$c_i=\{v_i^1,v_i^2,v_i^3\}$ and let $h_i^1,h_i^2,h_i^3$ be the associated
\emph{incoming} half-edges, as defined above. We proceed to the following construction:
\begin{itemize}
  \item we set $\mathfrak{m}_0 = \mathfrak{t}$
  \item for $i$ from $1$ to $g$, let $\mathfrak{m}_i$ be the map obtained by
  the only gluing of $h_i^1,h_i^2,h_i^3$ in $\mathfrak{m}_{i-1}$ that produces
  a unicellular map. We let $v_i$ be the vertex created by that gluing
  operation.
\end{itemize}
We set $v_*=(v_1,\ldots,v_g)$, and $\Psi(\mathfrak{t},c_*) =
(\mathfrak{m}_g,v_*)$.
\end{definition}
We have:
\begin{theorem}
\label{thm:bij}
The map $\Psi$ is a well-defined application:
$
\Psi : \mathcal{T}_{g,n} \longrightarrow \mathcal{O}^*_{g,n}
$.
Moreover,\begin{itemize}
\item
for every $(\mathfrak{t},c_*)\in \mathcal{T}_{g,n}$,
one has $\Phi\circ\Psi(\mathfrak{t},c_*)=(\mathfrak{t},c_*)$.
\item
for every $(\mathfrak{m},v_*)\in \mathcal{O}^*_{g,n}$, one has
$\Psi\circ\Phi(\mathfrak{m},v_*)=(\mathfrak{m},v_*)$.
\end{itemize}
In other words, $\Psi$ and $\Phi$ are reciprocal bijections between
$\mathcal{T}_{g,n}$ and $\mathcal{O}^*_{g,n}$.
\end{theorem}
\begin{proof}
The fact that for each $i$, $\mathfrak{m}_i$ is a well
defined  unicellular map of genus $i$
is a consequence of the remark following Lemma~\ref{lemma:onlyonegluing}.
Now, observe that for each
$i$, the edges of the core of the map $\mathfrak{m}_i$ are also edges of the
tree $\mathfrak{r}(\{v_i^j,\:i\in\llbracket 1,g\rrbracket,\:j\in\llbracket 1,3
\rrbracket)\}$, as defined in subsection~\ref{subsection:skeleton} (this
inclusion being an equality for $i=g$).
Hence, the core of $\mathfrak{m}_i$ has only vertices of degree $\leq 3$, so
that $\mathfrak{m}_i$ is a dominant unicellular map of genus $i$. It follows
from Lemma~\ref{lemma:onlyonegluing} again that for each $i$, $v_i$ is an
intertwined node of $\mathfrak{m}\setminus v_g\ldots v_{i+1}$. Thus
$(\mathfrak{m}_g,v_*)\in O^*_{g,n}$, and $\Psi$ is well defined.

Now, the fact that for every $(\mathfrak{m},v_*)\in \mathcal{O}^*_{g,n}$, one
has $\Psi\circ\Phi(\mathfrak{m},v_*)=(\mathfrak{m},v_*)$ is a direct consequence of
the last statement of Proposition~\ref{prop:opening}. 

Finally, the fact that for
every $(\mathfrak{t},c_*)\in \mathcal{T}_{g,n}$, one has
$\Phi\circ\Psi(\mathfrak{t},c_*)=(\mathfrak{t},c_*)$ directly follows from
Lemma~\ref{lemma:onlyonegluing}.
\end{proof}

\subsection{Enumerative corollaries}

Theorem~\ref{thm:bij} reduces the enumeration of dominant unicellular maps,
hence the asymptotic enumeration of unicellular maps, to  the one of trees with
$g$ triples. The following lemma gets rid of the non-singularity asssumption.

\begin{lemma}
\label{lemma:singular}
Let $p\geq 3$ be a positive integer. Let $\mathfrak{t}_n$ be a rooted plane tree
with $n$ edges, chosen uniformly at random, and let $\{v_1,\ldots,v_p\}$ a set of
$p$ vertices of $\mathfrak{t}_n$ chosen uniformly at random conditionally to
$\mathfrak{t}_n$. Then: $$
\mathbb{P}\left(\{v_1,\ldots,v_p\} \mbox{ is singular}\right)
=O\left(n^{-1/2}\right)
$$
\end{lemma}
\begin{proof}
We adapt the arguments of section~\ref{sec:scheme} to the case of trees.
Let us consider the skeleton of $\{v_1,\ldots,v_p\}$ in $\mathfrak{t}_n$.
Recall that for fixed $p$ the number of such skeletons is finite.
For each possible skeleton $\mathfrak{s}$, we let $T_\mathfrak{s}(z)$ be
the generating series of trees with $p$
marked vertices $\{v_1,\ldots,v_p\}$ of skeleton $\mathfrak{s}$.
All such trees are obtained by replacing
the edges of $\mathfrak{s}$ by elements of $\mathfrak{T}$. 
Hence, using the double-rooting argument as at the end of
subsection~\ref{subsec:doublerooting} (i.e. writing the generating series of
rooted trees with $p$ marked vertices whose skeleton has $k$ edges and carries
an additionnal oriented edge), we have: 
$$\sum_{\mathfrak{s}\in\mathcal{X}^p_k}
2k\cdot T_{\mathfrak{s}}(z) = 
\sum_{\mathfrak{s}\in\mathcal{X}^p_k}
2\frac{zd}{dz}T(z)^{k}
$$
which gives:
$$
T_k(z) = \frac{1}{k}\sum_{\mathfrak{s}\in\mathcal{X}^p_k}
\frac{zd}{dz}T(z)^{k}
$$
where $T_{k}(z)=\sum_{\mathfrak{s}\in\mathcal{X}^p_k}
T_{\mathfrak{s}}(z)$ is the generating series of rooted trees with $p$
marked vertices whose skeleton has $k$ edges.
Hence for each $k$, $T_{k}(z)$ is an algebraic
series of singular exponent $-\frac{k}{2}-1$, and the asymptotic regime is 
dominated by those skeletons that have the maximal number of edges. 

Now, among the finite set of possible skeletons, those that maximize the number of edges are those in which the $v_i$'s are leaves, and the other vertices have degree $3$, i.e. 
those where the set $\{v_1,\ldots,v_p\}$ is non-singular.
The lemma then follows by 
extracting the $n$-th
coefficient of the series $T_k$ by transfer theorems (\cite{FlOd}), as in the proof of Lemma~\ref{lemma:dominant}.
\end{proof}

Recalling that trees are counted by Catalan numbers, the number of trees with
$n$ edges and $g$ distinguished disjoint subsets of three vertices is: 
\begin{eqnarray*}
\frac{1}{n+1}{2n\choose n}{n+1\choose 3,\ldots,3,n+1-3g} &=&
\frac{(2n)!}{6^gn!(n+1-3g)!}  \\
&=& \frac{1}{6^g\sqrt{\pi}}n^{-3/2+3g} 4^n
\left(1+O\left(\frac{1}{n}\right)\right)
\end{eqnarray*}
Hence, from the previous lemma, the number of trees with $g$ triples and $n$
edges is: 
\begin{eqnarray}          
\label{eq:Tgn}
\left|\mathcal{T}_{g,n}\right| =
\frac{1}{6^g\sqrt{\pi}}n^{-3/2+3g} 4^n
\left(1+O\left(\frac{1}{\sqrt{n}}\right)\right)
\end{eqnarray}
Now, putting the previous results together gives:
\begin{eqnarray*}
\left|\mathcal{T}_{g,n}\right| &=& 
\left|\mathcal{O}^*_{g,n}\right| \mbox{ [Theorem~\ref{thm:bij}]}\\
&=& 2^gg!\left|\mathcal{U}^*_{g,n}\right| \mbox{
[Proposition~\ref{prop:2^gg!}]}\\ 
&=&2^gg!\left|\mathcal{U}_{g,n}\right| \left(
1+O(n^{-1/2})\right) \mbox{ [Lemma~\ref{lemma:dominant}]}
\end{eqnarray*}
and Equation~\ref{eq:Tgn} gives our first corollary 
(this result has been known for some time from other techniques, see for
example \cite{Goupil-Schaeffer}):
\begin{corollary}
Fix $g\geq 1$. The number of unicellular maps with $n$ edges satisfies, when
$n$ tends to infinity:
 $$
\left|\mathcal{U}_{g,n}\right|=
\frac{n^{3g-\frac{3}{2}}}{12^gg!\sqrt{\pi}} 4^n
\left(1+O\left(\frac{1}{\sqrt{n}}\right)\right)
$$
\end{corollary}

Since we are dealing from the beginning with dominant maps, it seems
unnecessary to try to get rid of Lemma~\ref{lemma:dominant}, and to try to
perform the exact enumeration of trees with $g$ triples.
Let us mention, however, a very simple case where our bijection applies and
enables to perform exact enumeration. This gives an easy bijective proof of a known
formula of Lehman and Walsh (precisely, the next corollary is a special case of
Equation~9 in \cite{Lehman-Walsh-genus-I} ; see also 
\cite{Bacher-Vdovina,Goupil-Schaeffer}). A \emph{triangulation} is a map where all faces have
degree $3$. Triangulations with one vertex are in bijection, by classical duality, with maps with one face and all vertices of degree $3$: these maps are exactly
our \emph{dominant schemes}.
\begin{corollary}
Let $\mathcal{T}^*_g$ be the set of rooted plane trees with $6g-3$ edges,
$3g$ leaves and $3g-2$ vertices of degree $3$. The bijection $\Phi$
specializes to a bijection between the set of dominant schemes
equipped with an opening sequence and the set of pairs $(\mathfrak{t},c_*)$,
where $\mathfrak{t}\in\mathcal{T}^*_g$ and $c_*$ is an ordered  partition of the
leaves of $\mathfrak{t}$ in sets of three elements.

The number of dominant schemes of genus $g$ (equivalently, of
rooted triangulations of genus $g$ with one vertex) is:
$$
\left|\mathcal{S}^*_{g}\right|=\frac{2(6g-3)!}{12^gg!(3g-2)!}
$$
\end{corollary}
\begin{proof}
Let $(\mathfrak{s},v_*)$ be a dominant scheme equipped with an opening sequence,
and let $\mathfrak(t,c_*)=\Phi(\mathfrak{s},v_*)$. By definition, all the
vertices of $\mathfrak{s}$ are nodes, and have degree $3$. The nodes which
belong to $v_*$ will give birth to three leaves in $\mathfrak{t}$,
whereas the other nodes will remain vertices of degree $3$ in $\mathfrak{t}$.
Consequently, $\mathfrak{t}$ is an element of $\mathcal{T}^*_g$.  Conversely,
it is clear that for any $\mathfrak{t}\in\mathcal{T}^*_g$, every partition
$c_*$ of the leaves of $\mathfrak{t}$ in sets of three elements is non-singular.
Moreover, such a pair $\mathfrak{t},c_*$ being given,
$\Psi(\mathfrak{t},c_*)$ is an opened map with all vertices of degree $3$,
i.e. a dominant scheme with an opening sequence. This proves the first
assertion of the corollary.

We now compute the cardinality of $\mathcal{T}^*_g$. First observe that an
element of $\mathcal{T}^*_g$ has $2(6g-3)$ half-edges, $3g$ of them
begin attached to a leaf. Hence counting trees which have at the same time a
root and a distinguished leaf gives:
$$
3g\cdot \left|\mathcal{T}^*_g\right| = 2(6g-3)\cdot
\left|\widehat{\mathcal{T}}^*_g\right|
$$
where $\widehat{\mathcal{T}}^*_g$ is the set of elements of $\mathcal{T}^*_g$
which are rooted at a leaf. Now, by removing the root edge, one
easily sees that $\widehat{\mathcal{T}}^*_g$ is in bijection with rooted binary
trees with $3g-2$ inner nodes (and $6g-4$ edges), so that:
$
\left|\widehat{\mathcal{T}}^*_g\right|=
\frac{(6g-4)!}{(3g-1)!(3g-2)!}
$. This gives:
$$
\left|\mathcal{T}^*_g\right|=
\frac{2(6g-3)!}{(3g)!(3g-2)!}
$$

Finally, the number of ways to partition the leaves in sets of three elements
is $\frac{(3g)!}{(3!)^g}$, so that the number of dominant schemes of genus $g$
with an opening sequence is:
$$
\frac{(3g)!}{(3!)^g} \cdot \left|\mathcal{T}^*_g\right| = 
\frac{2(6g-3)!}{6^g(3g-2)!}
$$
Applying Proposition~\ref{prop:2^gg!} and dividing by $2^g g!$ gives the second
statement of the corollary.%
\end{proof}

\section{The case of labelled unicelullar maps: labelled trees and ISE.}

\label{sec:labelled}

\subsection{The Marcus-Schaeffer bijection and the volume constant of maps of
genus $g$} 
Our interest for unicellular maps originally comes from the fact that
labelled unicellular maps are in bijection with all maps. We begin with a short reminder
of this fact.

\begin{definition}
A \emph{labelled unicellular map of genus $g$} is a rooted unicellular map
$\mathfrak{m}$ of genus $g$, together with an application:
$$
l: \{\mbox{vertices of }\mathfrak{m}\}\longrightarrow \mathbb{Z}
$$
such that:
\begin{itemize}
  \item[i.] $l(\mbox{root})=0$
  \item[ii.] if two vertices $v_1$ and $v_2$ are linked by an edge in
  $\mathfrak{m}$, then $l(v_1)-l(v_2)$ is an element of $\{-1,0,+1\}$.
\end{itemize}
The set of labelled unicellular maps of genus $g$ is denoted
$\mathcal{L}_{g}$. 
The set of labelled unicellular maps of genus $g$, which are
moreover dominant (in the sense of the previous sections) is denoted
$\mathcal{L}^*_{g}$.
\end{definition}
We also let $\mathcal{Q}_{g,n}$ be the set of rooted bipartite
quadrangulations with $n$ faces, and $\mathcal{Q}^\bullet_{g,n}$ be the set
of rooted bipartite quadrangulations with $n$ edges which carry an additional
distinguished vertex. Since a quadrangulation of genus $g$ with $n$ faces has
$n+2-2g$ vertices, one has:
$|\mathcal{Q}^\bullet_{g,n}|=(n+2-2g)|\mathcal{Q}_{g,n}|$.
Our motivation for studying bipartite quadrangulations is the classical
bijection of Tutte (\cite{Tutte:census}), which says that they are in
bijection with (general) rooted maps, and in particular that the number of
rooted maps with $n$ edges and genus $g$ satisfies:
$$
\left|\mathcal{M}_{g,n}\right|=\left|\mathcal{Q}_{g,n}\right|
$$
The following theorem is a
reminder of the known bijections between quadrangulations and labelled unicellular maps.
\begin{othertheorem}[\cite{MaSc,ChMaSc}]
There exists a bijection:
$$\tau: \mathcal{Q}^\bullet_{g,n} \longrightarrow
\{0,1\}\times\mathcal{L}_{g,n} $$
such that for every quadrangulation $\mathfrak{q}$ of pointed vertex $\bullet$,
and such that $\tau(\mathfrak{q})=(\epsilon,\mathfrak{l})$, there is a bijection:
$\nu: \{\mbox{vertices of }
\mathfrak{q}\}\setminus\{\bullet\}\rightarrow\{\mbox{vertices of }
\mathfrak{l}\}$ such that for every 
non-root vertex $v$ of $\mathfrak{q}$, 
$l(\nu(v))-\min_w\{l(\nu(w))\}+1$ is the graph-distance between $v$ and
$\bullet$ in $\mathfrak{q}$.

Moreover, one has, when $n$ tends to infinity:
\begin{eqnarray}
\label{eq:dominant}
\frac{|\mathcal{L}^*_{g,n}|}{|\mathcal{L}_{g,n}|} = 1-O\left(n^{-1/4}\right)
\end{eqnarray}
and 
$\displaystyle
|\mathcal{Q}^\bullet_{g,n}| =  t_g n^{\frac{5g-3}{2}} 12 ^n
\left(1+O\left(n^{-1/4}\right)\right) $
for some positive constant $t_g$.
\end{othertheorem}

We call $t_g$ the \emph{volume constant} of maps of genus $g$. Observe that we
have:
\begin{eqnarray}
\label{eq:deftg}
|\mathcal{M}_{g,n}| \sim  t_g n^{\frac{5(g-1)}{2}} 12 ^n.
\end{eqnarray}
In \cite{ChMaSc}, an expression for $t_g$ is found which involves a finite sum
which is not easy to compute in practice. In the rest of this section, we give
another  proof of the last statement of the theorem that uses our bijection.
In particular, we obtain  another expression for the constant $t_g$,
related to random trees and ISE.

\begin{definition}\label{def:welllabelledtriples}
A \emph{labelled tree with $g$ well-labelled triples} is a tree with $g$
triples $(\mathfrak{t},c_*)$, together with an application 
$$
l: \{\mbox{vertices of }\mathfrak{t}\}\longrightarrow \mathbb{Z}
$$
such that:
\begin{itemize}
  \item[i.] $l(\mbox{root})=0$
  \item[ii.] if two vertices $v_1$ and $v_2$ are linked by an edge in
  $\mathfrak{t}$, then $l(v_1)-l(v_2)$ is an element of $\{-1,0,+1\}$.
  \item[iii.] for every $i\in\llbracket 1, g \rrbracket$, if
  $c_i=\{v_i^1,v_i^2,v_i^3\}$, then $l(v_i^1)=l(v_i^2)=l(v_i^3)$.
\end{itemize}
$\mathcal{W}_{g}$ is the set of labelled trees with $g$ well-labelled triples.
\end{definition}

It is clear that the bijection $\Phi$ extends to the case of
labelled unicellular maps: the only thing to check is, before gluing three
vertices, that they have the same label. This is exactly done in the definition
above. Hence:
\begin{corollary}
\label{cor:bijlabelled}
The applications $\Phi$ and $\Psi$ extend to bijections between labelled trees
with $g$ well-labelled triples and $n$ edges, and dominant labelled unicellular
maps with $n$ edges equipped with an opening sequence. One has:
$$
\left|\mathcal{W}_{g,n}\right|= 2^gg!
\left|\mathcal{L}^*_{g,n}\right|
$$
\end{corollary}

A \emph{labelled tree} is a rooted plane tree with an application $l$ that
satifies the properties $i$ and $ii$ of Definition~\ref{def:welllabelledtriples}. Equivalently, a
labelled tree is a rooted plane tree with an application $\{edges\}
\longrightarrow \{-1,0,+1\}$, which encodes the variation of the label when
crossing this edge coming from the root. Since a tree has no cycle, this
application has no constraint to satisfy, and the number of rooted labelled
trees with $n$ edges is:
$\displaystyle
\frac{3^n}{n+1}{2n\choose n}
$.
In what follows, the label function of a rooted labelled tree is always denoted
$l$.

\begin{lemma}
\label{lemma:enumskeleton}
Let $\mathcal{R}_{g,n}$ be the set of rooted labelled trees with $n$ edges
which carry (non necessarily distinct) distinguished vertices
$v_1,v_2,\ldots,v_{3g}$ such that for all $i\in\llbracket 1,g \rrbracket :
l(v_{3i-2})=l(v_{3i-1})=l(v_{3i})$. Then one has:
$$
\left|\mathcal{R}_{g,n}\right|
=c_g n^{\frac{5g-3}{2}} 12^n
\left(1+O\left(n^{-1/4}\right)\right)
$$
for some positive constant $c_g$.

Moreover:
$$
\left|\mathcal{W}_{g,n}\right|
=\frac{\left|\mathcal{R}_{g,n}\right|}{6^g}
\left(1+O\left(n^{-1/4}\right)\right)
$$
\end{lemma}
Our proof of Lemma~\ref{lemma:enumskeleton} follows the method introduced in
\cite{ChMaSc}, where it was applied directly to labelled schemes (instead of
considering trees and their skeleton). We need first another lemma:
\begin{lemma}
\label{lemma:treewalks}
For all $i\geq 0$, let $\mathfrak{T}_{i}$ be the set of triple
$(\mathfrak{t},\nu,l)$ such that $(\mathfrak{t},l)$ is a rooted labelled tree,
 $(\mathfrak{t},\nu)\in\mathfrak{T}$, and
$l(\nu)=i$.
Then the generating series 
$\displaystyle
N_i(z)=\sum_{(\mathfrak{t},\nu,l)\in\mathfrak{T}_{i}}z^{|\mathfrak{t}|}$
satisfies: 
\begin{eqnarray}
\label{eq:Mi}
N_{i}(z) = (B(z)-\mathbbm{1}_{i=0}) \left[U(z)\right]^{i}
\end{eqnarray}
where $B$ and $U$ are two algebraic series of radius of convergence $1/12$, with
singular expansion at that point:
\begin{itemize}
  \item $B(z)= C_1(1-12z)^{-1/4}  + O(1)$  
  \item $U(z)= 1-C_2(1-12z)^{1/4} +O\left(\sqrt{1-12z}\right)$
\end{itemize}
for some constants $C_1,C_2>0$.
\end{lemma}
\begin{proof}[Proof of Lemma~\ref{lemma:enumskeleton}]
We admit Lemma~\ref{lemma:treewalks}. Let $(\mathfrak{t}_n,v_*)$ be an element of $\mathcal{R}_{g,n}$, and 
let $M=\left|\{l(v_i)\:, \: 1\leq i \leq 3g \}\right|-1$. The \emph{labelled skeleton} of
$(\mathfrak{t}_n,v_*)$ is the pair $(\mathfrak{s},\lambda)$ where
$\mathfrak{s}$ is the (unlabelled) skeleton defined above, and $\lambda$ is the
the unique surjective application:
$\displaystyle
\{\mbox{vertices of }\mathfrak{s}\} \longrightarrow \llbracket 0, M\rrbracket
$
that satisfies:
\begin{eqnarray}
\label{eq:lambda}
l(v)> l(w) \Leftrightarrow
\lambda(v)> \lambda(w) \mbox{ and } l(v)=l(w) \Leftrightarrow\lambda(v)=
\lambda(w)
\end{eqnarray}
We let $\mathcal{Y}_k$ be the set of all pairs
$(\mathfrak{s},\lambda)$ such that $\mathfrak{s}$ has $k$ edges and there exists
an element of $\mathcal{R}_g$ of labelled skeleton $(\mathfrak{s},\lambda)$.
Oberve that $\mathcal{Y}_k$ is finite.

Now, let
$(\mathfrak{s},\lambda)$ be a labelled skeleton. 
A \emph{compatible labelling} of $(\mathfrak{s},\lambda)$ is an application 
$l: \{\mbox{vertices of }\mathfrak{s}\} \longrightarrow \mathbb{N}$ that
satisfies Equation~\ref{eq:lambda}. Observe that all such labellings are of the
form: $$
l(v) = \sum_{i=1}^{\lambda(v)} \delta_i
$$
for some $\delta\in\left(\mathbb{N}\setminus\{0\}\right)^M$. 

We will use again the double-rooting argument of subsection~\ref{subsec:scheme}.
We let $R_k(z)$ be the generating series of elements of $\mathcal{R}_g$ whose
skeleton has $k$ edges, so that $2kR_k(z)$ is the generating series of
elements of $\mathcal{R}_g$ whose skeleton has $k$ edges and carries an
additional distinguished oriented edge. Now, all such objects can be obtained in
a unique way as follows:%
\begin{itemize}
  \item first, choose a labelled skeleton
  $(\mathfrak{s},\lambda)\in\mathcal{Y}_k$
  \item then, choose a labelling $l$ compatible with $\lambda$ (equivalently,
  an element $\delta\in\left(\mathbb{N}\setminus\{0\}\right)^M$)
  \item for each edge $e$ of $\mathfrak{s}$, let $l(e_+)\geq l(e_-)$
  denote the labels of its two extremities. Choose a rooted labelled plane
  tree $\mathfrak{t}_e\in\mathfrak{T}_{l(e_+)-l(e_-)}$.
  Shift the labels of that tree by the quantity $l(e_-)$, so that the root edge
  (resp. the marked vertex) of $\mathfrak{t}_e$ has label $l(e_-)$ (resp. $l(e_+)$).
  \item replace each edge $e$ by the associated tree $\mathfrak{t}_e$, with
  the convention of Figure~\ref{fig:supertree}.
  \item shift all labels in order that the root vertex has label $0$.
  \item distinguish an oriented edge as the root of the map. The distinguished
  oriented edge of  its skeleton is given by the root of $\mathfrak{s}$.
\end{itemize}
For each $(\mathfrak{s,\lambda})\in\mathcal{Y}_k$, we set
$E^{\mathfrak{s,\lambda}}_{\neq}=\{\mbox{edges of }\mathfrak{s},\:l(e_-)\neq
l(e_+)\}$, $E^{\mathfrak{s,\lambda}}_{=}=\{\mbox{edges of
}\mathfrak{s},\:l(e_-)= l(e_+)\}$, and
$E^{\mathfrak{s,\lambda}}=E^{\mathfrak{s,\lambda}}_{\neq}\cup
E^{\mathfrak{s,\lambda}}_=$.
From the construction above, the generating series $2k\cdot R_k(z)$ can be
written:
\begin{eqnarray*}
 2k\cdot R_k(z)
&=& 
2\cdot\frac{zd}{dz}\left(
\sum_{(\mathfrak{s},\lambda)\in\mathcal{Y}_k}
\sum_{~l
{\mbox{ \scriptsize compatible }}} \prod_{e\in E^{\mathfrak{s,\lambda}}} 
N_{l(e_+)-l(e_-)}(z) \right)
\\ 
&=& 
2\cdot\frac{zd}{dz}\left(
\sum_{(\mathfrak{s},\lambda)\in\mathcal{Y}_k}
\sum_{~l {\mbox{ \scriptsize compatible }}} \prod_{e\in
E^{\mathfrak{s,\lambda}}_{\neq}} \left(B(z) U(z)^{l(e_+)-l(e_-)} \right)
\prod_{e\in E^{\mathfrak{s,\lambda}}_{=}} \left(B(z)
-1 \right) \right)
\\ 
&=& 
2\cdot\frac{zd}{dz}\left(
\sum_{(\mathfrak{s},\lambda)\in\mathcal{Y}_k}
B(z)^{|E^{\mathfrak{s,\lambda}}_{\neq}|}(B(z)-1)^{|E^{\mathfrak{s,\lambda}}_{=}|}
\sum_{\delta_1,\ldots\delta_{M}>0} U(z)^{
\delta_{\lambda(e_-)+1}+\ldots+\delta_{\lambda(e_+)}}\right)
\\ &=& 
2\cdot\frac{zd}{dz}\left(
\sum_{(\mathfrak{s},\lambda)\in\mathcal{Y}_k}
B(z)^{|E^{\mathfrak{s,\lambda}}_{\neq}|}(B(z)-1)^{|E^{\mathfrak{s,\lambda}}_{=}|}
\prod_{i=1}^{M}
\frac{U(z)^{d_i^{\mathfrak{s,\lambda}}}}{1-U(z)^{d_i^{\mathfrak{s,\lambda}}}}\right)
\end{eqnarray*}
where 
$\displaystyle d_i^{\mathfrak{s,\lambda}} = |\{e \mbox{ edge of }\mathfrak{s} :
\lambda(e_-)<i\leq \lambda(e_+)\}|$. 
 Observe that for all $i$, $d_i^{\mathfrak{s,\lambda}}$ is
positive, which implies finally that $R_k(z)$ has singular expansion:
\begin{eqnarray}
\label{eq:singularskeleton}
R_k(z) =& & \frac{1}{k}\sum_{(\mathfrak{s},\lambda)\in\mathcal{Y}_k}
-\frac{k+M}{2}
\left(\prod_i\frac{1}{d_i^{\mathfrak{s,\lambda}}}\right) C_1^k C_2^M
(1-4z)^{-\frac{k+M}{4}-1} \\
& & +O\left( (1-4z)^{-\frac{k+M}{4}-\frac{3}{4}}\right)
\end{eqnarray}
The greatest contribution is therefore realized by elements of
$\mathcal{R}_g$ whose skeleton maximizes the quantity $k+M$.
Now, the maximal value of $k$ is $6g-3$: it is reached when the skeleton is a
tree in which all the  $v_i$'s  are distinct and have degree $1$, while the
other vertices have degree $3$. The maximal value of $M+1$ is $4g-2$, and is reached when all
the labels are distinct (the $3g-2$ labels of the inner nodes, plus the $g$
labels of the marked vertices). Hence the maximal value of $k+M$ is $10g-6$, which corresponds to a critical exponent 
$\frac{1}{2}-\frac{5}{2}g$.

One the one hand, this implies with transfer theorems (\cite{FlOd}) that:
$$\left|\mathcal{R}_{g,n}\right|=c_g n^{\frac{5g-3}{2}}12^n
\left(1+O\left(n^{-1/4}\right)\right),
\mbox{where }c_g =
\frac{C_1^{6g-1}C_2^{4g-1}}{\Gamma\left(\frac{5g-1}{2}\right)}
\sum_{\mathfrak{s},\lambda}\frac{1}{|\mathfrak{s}|}\prod_i\frac{1}{d_i}
$$
and the sum is taken over those $(\mathfrak{s},\lambda)$ for which $k+M=10g-6$.

One the other hand, the dominating terms exactly correspond to the case where
$\{v_1,\ldots,v_{3g}\}$ has cardinality $3g$ and is non-singular. Up to
forgetting the order of the vertices of each triple
$\{v_{3i-2},v_{3i-1},v_{3i}\}$ (which induces a factor $\frac{1}{(3!)^g}$)
these are the elements of $\mathcal{W}_g$. Since the asymptotic expansions
involve only exponents which are multiple of $\frac{1}{4}$, the second statement of the lemma follows by transfer theorems.
\end{proof} 
\begin{proof}[Proof of Lemma~\ref{lemma:treewalks}]
A \emph{Motzkin walk of increment $i$} is a finite walk on the integers, 
starting at $0$, having steps in $\{-1,0,+1\}$ and ending at position $i$.
All elements of $\mathfrak{T}_i$ can be constructed in a unique way as follows.
First, choose a Motzkin walk of increment $i$ (if $i=0$ it has to have positive
length). If this walk has $m$ steps, draw a chain of $m+1$ vertices linked
by $m$ edges, and assign to the $j$-th vertex of the chain the $j$-th label of
the walk. Finally, attach one planar labelled tree on each of the $2m$ corners
of this walks. The first edge (resp. the last vertex) of the chain gives the
root (resp. the marked vertex) of the obtained tree. Hence, if $M_i(t)$ is the
generating series of Motzkin walks of increment $i$, one has:
$N_i(z)=M_i(zC(z)^2)$ where $C(z)=\frac{1-\sqrt{1-12z}}{6z}$ is the generating series of rooted
labelled trees.

An \emph{excursion} is a Motzkin walk of increment $0$ that takes its values in
$\mathbb{N}$. Thanks to a decomposition at the first return to $0$, the
generating series $E(t)$ of excursions satisfies: 
$\displaystyle
E(t)=1+t E(t)+t^2E(t)^2
$.
Moreover, decomposing a walk at its passages at $0$, one sees that the series
$M_0$ is related to $E$ by: $\displaystyle
M_0(t)= \frac{1}{1-t-2t^2E(t)}.
$
Then, $M_i$ is easily computed thanks to a last passage
decomposition:
$$
M_i(t)=M_0(t)[tE(t)]^i
$$
Hence we have proved Equation~\ref{eq:Mi}, with: 
$t(z)=zC(z)^2$, $B(z)=M_0(t(z))$, and $U(z)=t(z)E(t(z))$. 
Observe the term $-\mathbbm{1}_{i=0}$, which we need to exclude the case of the
empty walk, which is counted in the series $M_0$ but is irrelevant in our
decomposition, since in an element of $\mathfrak{T}$ the marked vertex and the
root cannot coincide. Now, a computation gives: $$
U(t)=\frac{1-t-\sqrt{(t+1)(1-3t)}}{2t}
\mbox{ and }
M_0 = [(t+1)(1-3t)]^{-1/2}
$$
Finally, we have $1-3t(z)=2\sqrt{1-12z}+O(1-12z)$, which
ends the proof of the lemma, giving $C_1=\frac{\sqrt{3}}{2\sqrt{2}}$ and
$C_2=\sqrt{6}.$
\end{proof}

Recall that we have from Corollary~\ref{cor:bijlabelled}:
$\displaystyle
\left|\mathcal{L}^*_{g,n}\right|=\frac{1}{2^gg!}\left|\mathcal{W}_{g,n}\right|
$
so that the lemma implies:
$$
\left|\mathcal{L}^*_{g,n}\right|=\frac{c_g}{12^gg!}n^{\frac{5g-3}{2}}12^n
\left(1+O\left(n^{-1/4}\right)\right) $$

We now express the constant $c_g$ in terms of labelled trees.
Let $\mathfrak{t_n}$ be a random rooted labelled plane tree with $n$ edges
chosen uniformly at random, and let $v_1,v_2\ldots,v_{3g}$ be $3g$ vertices of
$\mathfrak{t}_n$ chosen independently and uniformly at random.
Then one has by definition of the uniform probability:
\begin{eqnarray*}
\mathbb{P}\Big(
\forall i\in\llbracket 1,g \rrbracket :
l(v_{3i-2})=l(v_{3i-1})=l(v_{3i})
\Big)&=& \left|\mathcal{R}_{g,n}\right|
\left[\frac{3^{n}}{n+1}{2n \choose n}\cdot n^{3g}
\right]^{-1}\\
&=& c_g \sqrt{\pi} \cdot  n^{-g/2} \left(1+O\left(n^{-1/4}\right)\right)
\end{eqnarray*}

This gives our second Theorem, linking the volume constant $t_g$ to random
trees:
\begin{theorem}
\label{thm:proba}
Let $\mathfrak{t_n}$ be a random rooted labelled plane tree with $n$ edges
chosen uniformly at random, and let $v_1,v_2\ldots,v_{3g}$ be $3g$ vertices of
$\mathfrak{t}_n$ chosen independently and uniformly at random.
Then we have\begin{eqnarray}
t_g = \frac{2}{12^gg!\sqrt{\pi}} \lim_{n\rightarrow\infty} 
n^{g/2}\mathbb{P}\Big(
\forall i\in\llbracket 1,g \rrbracket :
l(v_{3i-2})=l(v_{3i-1})=l(v_{3i})
\Big)
\end{eqnarray}
where $t_g$ is the volume constant defined by Equation~\ref{eq:deftg}.
\end{theorem}

\subsection{Expressing the asymptotic number of maps in terms of ISE.}

We now relate the constant $t_g$ to the random measure ISE, which is
well-known to appear as a natural limit of many classes of random labelled trees
models. The ISE (for Integrated Superbrownian Excursion) 
is introduced in~\cite{Aldous-treebasedmodels}. We follow the notation in
\cite{MBM-Janson}.

Let $\mathfrak{t}_n$ be a labelled tree with $n$ edges chosen
uniformly at random. For all $k\in \mathbb{Z}$ we note $X_n(k)$ the 
(random) number of nodes of $\mathfrak{t}_n$ of label $k$.
Now, let as before $v_1,v_2\ldots,v_{3g}$ be $3g$ vertices of
$\mathfrak{t}_n$ chosen independently and uniformly at random.
We have:
\begin{eqnarray}\label{eq:prob1}
&&\mathbb{P}\Big(
\forall i\in\llbracket 1,g \rrbracket :
l(v_{3i-2})=l(v_{3i-1})=l(v_{3i})
\Big)\\
&=&\mathbb{E}\left[\mathbb{P}\Big(
\forall i\in\llbracket 1,g \rrbracket :
l(v_{3i-2})=l(v_{3i-1})=l(v_{3i})
 \:\Big|\: \mathfrak{t}_n
\Big)\right] \\
&=&\mathbb{E}\left[\left(
\frac{\sum_{k\in \mathbb{Z}} X_n(k)^3}{(n+1)^{3}} \right)^g\right]\label{eq:prob3}
\label{eq:probasomme}
\end{eqnarray}

We let $f_n:=\displaystyle\left(\frac{X_n\left(\left[\gamma^{-1}n^{1/4}x\right]\right)}{\gamma n^{3/4}}\right)_{x\in \mathbb{R}}$, where $[\cdot]$ denotes the integer part, and where $\gamma=2^{-1/4}3^{1/2}$. The function $f_n$ is an element of the set $\mathcal{D}_0(\mathbb{R})$ of \emph{c\`adl\`ag functions} on $\mathbb{R}$ (i.e. right-continuous functions with left-hand limits) which tend to $0$ at $\pm \infty$. We equip the set $\mathcal{D}_0(\mathbb{R})$ with the topology of the \emph{uniform norm}, denoted $\|. \|$.
The following theorem is due to Bousquet-M\'elou and Janson:
\begin{othertheorem}[\cite{MBM-Janson}]
Let $\mu_{{\scriptscriptstyle\mathrm{ISE}}}$ be the 1-dimensional ISE
measure. Then $\mu_{{\scriptscriptstyle\mathrm{ISE}}}$ has almost surely a continuous density 
$f_{{\scriptscriptstyle\mathrm{ISE}}}(x)$. Moreover one has
when $n$ tends to infinity:
\begin{eqnarray}
\label{eq:convISE}
f_n(x) \longrightarrow
f_{{\scriptscriptstyle\mathrm{ISE}}}(x)
\end{eqnarray}
in the sense of weak convergence in the space
$\mathcal{D}_0\left(\mathbb{R}\right)$ equipped with the uniform topology.
\end{othertheorem}
\noindent {\bf Remark.} In the paper \cite{MBM-Janson}, the theorem is not stated exactly in this form. Precisely, it is shown that 
$g_n(x) \rightarrow f_{{\scriptscriptstyle\mathrm{ISE}}}(x)$, where $g_n$ is the affine by parts process that coincides with $f_n$ at each point $x$ of the form $\frac{j}{\gamma^{-1}n^{1/4}}$ for integer $j$, and which is affine 
on the intervals between these points. 
But since
$\|f_n-g_n\|\leq w_{g_n}(\gamma n^{-1/4})$, where $w_{g_n}$ is the modulus of continuity of $g_n$, and since it is shown in \cite{MBM-Janson} that $f_{{\scriptscriptstyle\mathrm{ISE}}}$ has almost surely compact support, the theorem, as stated here, follows easily.


We can now express the constant $t_g$ as a functional of ISE:
\begin{theorem}
\label{thm:ISE}
The constant $t_g$ can be expressed as follows:
$$
t_g =
\frac{2}{2^{5g/2}g!\sqrt{\pi}}\mathbb{E}\left[\left(\int_{-\infty}^\infty
f_{{\scriptscriptstyle\mathrm{ISE}}}(x)^3dx\right)^g\right] $$
\end{theorem}
We let $\displaystyle
W_n=\int_{-\infty}^\infty f_n(x)^3dx$ and
$\displaystyle W=\int_{-\infty}^\infty f_{{\scriptscriptstyle\mathrm{ISE}}}(x)^3dx$. Observe that we have 
$W_n=\frac{1}{\gamma^2 n^{5/2}}\sum_{k\in \mathbb{Z}}X_n(k)^3$, so we obtain from Theorem~\ref{thm:proba} and  Equations~\ref{eq:prob1}--~\ref{eq:prob3}:
\begin{eqnarray}\label{eq:tgintermediaire}
t_g= \frac{2 \cdot\gamma^{2g}}{12^gg!\sqrt{\pi}} \lim_{n\rightarrow \infty}  \mathbb{E}\left[W_n^g\right],
\end{eqnarray}
which, in passing, proves that the limit exists and is finite. 
Theorem~\ref{thm:ISE} is therefore a direct consequence of the following lemma:
\begin{lemma}
\label{lemma:moments}
For all $g\in\mathbb{N}$, one has when $n$ tends to infinity:
\begin{eqnarray*}
 && \mathbb{E}\left[W_n^g\right] \longrightarrow 
 \mathbb{E}\left[W^g\right] < \infty
\end{eqnarray*}
\end{lemma}
\begin{proof}[Proof of Lemma~\ref{lemma:moments}] \hspace{2cm}\\
$\bullet$
We first show that $W_n$ converges weakly to $W$. Indeed, if $F$ is a bounded continuous real function, one has for 
every $m>0$:
\begin{eqnarray*}
&&\left|\mathbb{E}\left[F(W_n)\right]-\mathbb{E}\left[F(W)\right]\right|\leq
\left|\mathbb{E}\left[F\left(\int_{-m}^m f_n(x)^3dx\right)\right]-
\mathbb{E}\left[F\left(\int_{-m}^m f_{{\scriptscriptstyle\mathrm{ISE}}}(x)^3dx\right)\right]\right|  \\
&& \hspace{1.5cm}+ 2\|F\| \cdot \mathbb{P}\left(\textrm{supp}(f)\nsubseteq[-m,m]\right)
+  2\|F\| \cdot \mathbb{P}\left(\textrm{supp}(f_n)\nsubseteq[-m,m]\right)
\end{eqnarray*}
where for any real function $g$, $\textrm{supp(g)}$ denotes the support of $g$.
Now, it follows from the results of \cite{ChassaingSc} on the convergence of the support of ISE that for all $\epsilon>0$, there exists $m>0$ such that the second and third terms are smaller than $\epsilon$ (for all $n$). Moreover, since the mapping $g\longrightarrow \int_{-m}^m g(x)^3 dx$ is continuous with respect to the uniform norm on $\mathcal{D}_0$, the weak convergence of $f_n$ to $f_{{\scriptscriptstyle\mathrm{ISE}}}$ (Theorem~\cite{MBM-Janson}) implies that for $n$ large enough, the first term is also smaller than $\epsilon$. Hence $\left|\mathbb{E}\left[F(W_n)\right]-\mathbb{E}\left[F(W)\right]\right|\leq 3 \epsilon$ for $n$ large enough, i.e. $W_n$ converges weakly to $W$. 

\noindent$\bullet$ Again, we fix $\epsilon>0$. We know from Equation~\ref{eq:tgintermediaire} that for all $K$, $\mathbb{E}\left[W_n^K\right]$ tends to a \emph{finite} value $w_K$ when $n$ tends to infinity. Now, we have:
\begin{eqnarray*}
\mathbb{E}\left[W^{g}\right] &=& 
\lim_M \mathbb{E}\left[W^{g}\wedge M\right]\ \mbox{ [monotone convergence]} \\
&=&\lim_M \lim_n\mathbb{E}\left[ W_n^{g}\wedge M\right]
\ \mbox{[}W_n\mbox{ converges weakly to }W\mbox{]}\\
&\leq& w_g < \infty.
\end{eqnarray*}
In particular, this implies that the quantity $\mathbb{E}\left[W^g\mathbbm{1}_{W>L}\right]$ tends to $0$ when $L$ tends to infinity. Moreover, we have from Chebichev inequality 
$\mathbb{E}\left[W_n^g\mathbbm{1}_{W_n>L}\right]\leq \frac{1}{L}\mathbb{E}\left[W_n^{g+1}\right]$ which, for fixed $L$,  is smaller than $\frac{1+w_{g+1}}{L}$ for $n$ large enough.

We now fix $L$ large enough so that both $\mathbb{E}\left[W^g\mathbbm{1}_{W>L}\right]$ and $\frac{1+w_{g+1}}{L}$ are smaller than $\epsilon$. Up to replacing $L$ by a greater value, we can assume that $L$ is not an atom of the law of $W$ (should there be any). Then the fact that $W_n$ converges weakly to $W$ implies that 
$\left|\mathbb{E}\left[W_n^g\mathbbm{1}_{W_n\leq L}\right]-\mathbb{E}\left[W^g\mathbbm{1}_{W\leq L}\right]\right|$ is smaller than $\epsilon$ for $n$ large enough. Putting the three terms together, we obtain that for $n$ large enough, 
$\left|\mathbb{E}\left[W_n^g\right]-\mathbb{E}\left[W^g\right]\right|$ is smaller than $3\epsilon$, and the lemma is proved.
\end{proof}

\section{Convergence of the profile}

\label{sec:profile}

A consequence of our bijection is an explicit characterisation of the limiting
profile and radius of bipartite quadrangulations of genus $g$ in terms of
ISE.

Fix $g\geq 1$. For $n\geq 0$, let $q_n$ be a rooted and pointed
bipartite quadrangulation of genus $g$ with $n$ faces, chosen uniformly at
random. For all $k$ and $n$, let $Y_n(k)$ be the number of vertices of $q_n$ at graph distance
$k$ from the pointed vertex. For all $n$, define the probability
measure:
$$
\mathfrak{p}_{q_n} = \frac{1}{n+2-2g}\sum_{k=0}^\infty  Y_n(k)
\delta_{\gamma n^{-1/4}k} $$
where $\delta_x$ is the Dirac measure at $x$. The probability measure $\mathfrak{p}_{q_n}$ is called the \emph{profile} of $q_n$. It is a random
variable with values in the space $\mathcal{M}_1$ of all probability measures
on $\mathbb{R}$. The space $\mathcal{M}_1$ is equipped with the topology of weak convergence. 

 The quantity 
$a_n=\max\{k, \: Y_n(k) \neq 0 \}$ is called the \emph{radius} of
$q_n$.

\subsection{Statement of the limit theorem.}

Let
$l$ and $r$ be, respectively, the left and right bounds of the support of ISE:
$$
[l,r]=\cap \{I,\:I \mbox{ interval such that } \mu_{\ISE}(I)=1\}.
$$
Recall that $l$ and $r$ are almost surely finite (\cite{MBM-Janson}).
The \emph{shifted ISE measure} is the probability measure $\overline\mu_\ISE$
defined (almost surely) by:
$$
\int_\mathbb{R} h(x)d\overline\mu_\ISE(x)=\int_\mathbb{R} h(x-l)d\mu_\ISE(x)
$$
for all bounded continuous $h:\mathbb{R}\rightarrow\mathbb{R}$.
We let $\overline{\mathcal{L}}_\ISE$ be the law of $\overline\mu_\ISE$ on
$\mathcal{M}_1$. 
If a probability measure $\mu\in\mathcal{M}_1$ has a continuous density $f_\mu$
with compact support, we set: 
$$W(\mu) = \int_\mathbb{R}f_\mu(x)^3dx
$$ 
Observe that $W(\overline\mu_\ISE)$ is well-defined and finite, almost surely.
\begin{definition}
We define the probability measure $\mathcal{L}^g$ on
$\mathcal{M}_1$ by the relation:
$$
d\mathcal{L}^g(\mu)= \frac{1}{Z_g}
 W(\mu)^g
d\overline{\mathcal{L}}_\ISE(\mu)
$$
where $Z_g=\mathbb{E}\left[W(\mu_\ISE)^g\right]$. 
\end{definition}
\noindent In others terms, $\mathcal{L}^g$ is such that, for all bounded and continuous
functionals $h: \mathcal{M}_1\rightarrow\mathbb{R}$, one has:
$$
\int_{\mathcal{M}_1} h(\mu)
d\mathcal{L}^g(\mu)= \frac{1}{Z_g}
\int_{\mathcal{M}_1} h(\mu)
 W(\mu)^g
d\overline{\mathcal{L}}_\ISE(\mu)
$$

Our last theorem characterizes the profile of large bipartite
pointed quadrangulations:
\begin{theorem}
\label{thm:profile}
When $n$ tends to infinity, the law $\mathcal{L}^g_n$ of $\mathfrak{p}_{q_n}$ converges to
$\mathcal{L}^g$, in the sense of weak convergence with respect to the topology
of weak convergence on $\mathcal{M}_1$.
Moreover, let $\mu^g$ be a random element of $\mathcal{M}_1$ with law
$\mathcal{L}^g$, and let $$a=\min\{|I|,\: I\mbox{ interval s.t. }\mu^g(I)=1\}$$
be the range of $\mu^g$. Then $a$ is almost surely positive and finite, and the
normalized radius $\displaystyle\frac{a_n}{\gamma^{-1} n^{1/4}}$ converges in
law to $a$.
\end{theorem}
The rest of this section is devoted to the proof of the theorem. For details on the convergence of probability measures, we refer the reader to \cite{Billingsley-convergence}.

\subsection{Proof of Theorem~\ref{thm:profile}.}

For any bounded continuous real function $h$, and probability measure $\nu\in\mathcal{M}_1$, we note 
$\langle h,\nu\rangle=\int_\mathbb{R} h(x)\nu(dx)$.
Then, in order to prove that $\mathcal{L}^g_n$ converges to $\mathcal{L}^g$, it is sufficient to prove that 
$\langle h,\mathfrak{p}_{\mathfrak{q}_n}\rangle$ converges in distribution to $\langle h,\mu^g\rangle$ for every bounded continuous function $h$ (see for example~\cite[p.~264]{Kallenberg}). Moreover, it is sufficient to take $h$ in a \emph{countable dense} subset of the set of all continuous bounded functions on $\mathbb{R}$ (this, to avoid problems handling events of null probability in forthcoming uses of Skorokhod's representation theorem). We fix such a function $h$.

Now, in order to  to prove that $\langle h,\mathfrak{p}_{\mathfrak{q}_n}\rangle$ converges in ditribution to $\langle h, \mu^g\rangle$, it is sufficient to prove that 
\begin{eqnarray}\label{eq:convergenceM1}
\mathbb{E}\left[H(\langle h,\mathfrak{p}_{\mathfrak{q}_n}\rangle)\right]\rightarrow\mathbb{E}\left[H(\langle h,\mu^g\rangle)\right]
\end{eqnarray}
for all bounded \emph{uniformly} continuous real functions $H$ (see~\cite{Billingsley-convergence}). As before, we fix such a function $H$ in a dense countable subset of bounded uniformly continous real functions.

\subsubsection*{\bf step 1: using the bijection}

We have by definition:
\begin{eqnarray*}
\mathbb{E}\left[H(\langle h,\mathfrak{p}_{\mathfrak{q}_n}\rangle)\right]
&=&\frac{1}{|\mathcal{Q}_{g,n}^\bullet|}\sum_{q_n\in\mathcal{Q}_{g,n}^\bullet}
H(\langle h,\mathfrak{p}_{q_n}\rangle)\\
&=& \frac{1}{|\mathcal{Q}_{g,n}^\bullet|}\sum_{q_n\mbox{\scriptsize
dominant}} H(\langle h,\mathfrak{p}_{q_n}\rangle) \left(1+O(n^{-1/4})\right)
\end{eqnarray*}
where the sum is taken only on those quadrangulations whose associated
labelled unicellular map is dominant (we have used
Equation~\ref{eq:dominant} and that $H$ is bounded). Now, by
Proposition~\ref{prop:2^gg!}, $$\sum_{q_n\mbox{\scriptsize dominant}}
H(\langle h,\mathfrak{p}_{q_n})\rangle=\frac{1}{2^gg!}\sum_{(q_n,v_*)}H( \langle h,\mathfrak{p}_{q_n}\rangle)$$ 
where the second sum is taken over all pairs $(q_n,v_*)$ where $q_n$ is a
dominant quadrangulation and $v_*$ is an opening sequence of the unicellular
map associated to $q_n$. 

Thanks to Corollary~\ref{cor:bijlabelled}, it is possible to reformulate this
last sum in terms of labelled trees. First, if $t$ is a labelled tree of size
$n$ we let, for all $k$, $X_t(k)$ be its number of vertices of label $k$. We
set $\lambda_t=\inf\{k,\:X_t(k)\neq 0 \}$ and we define the real probability
measure: $$
\mathfrak{p}_{t} =  \frac{1}{n+1}\sum_{k=0}^\infty
X_{t}(\lambda_t+k-1) \delta_{\gamma^{}n^{-1/4}k}
$$
Now, if $(t,c_*)$ is an element of $\mathcal{R}_{g,n}$, we define the measure:
$$\mathfrak{q}_{t,c_*} = \frac{\delta_0 -
2\sum_{i=1}^g\delta_{\gamma^{}n^{-1/4}(l(c_i)-\lambda_t+1)}}{n+1}.$$
Observe that from Theorem~\cite{MaSc}, the measure
$\frac{n+1}{n+2-2g}(\mathfrak{p}_{t_n}+\mathfrak{q}_{t_n,c_*})$ is exactly the
profile of the quadrangulation associated to $t_n$ via our bijection and 
Marcus-Schaeffer's bijection. Indeed
the correction measure $\mathfrak{q}_{t_n,c_*}$ accounts first for the fact that
during the gluing operation, each triple of marked vertices gives birth to only
one vertex of the unicellular map (so that two vertices disappear in the
operation), and then for the pointed vertex of label $0$, which is not present
in the labelled unicellular map, but is in the quadrangulation.

Then we have from Corollary~\ref{cor:bijlabelled}:
$$\sum_{q_n\mbox{\scriptsize
dominant}} H(\langle h,\mathfrak{p}_{q_n}\rangle) =\frac{2}{2^gg!}
\sum_{(t_n,c_*)\in\mathcal{W}_{g,n}}
H\left(
\frac{n+1}{n+2-2g}(\langle h,\mathfrak{p}_{t_n}\rangle+\langle h,\mathfrak{q}_{t_n,c_*}\rangle)
\right) 
$$
Now,  on the one hand, we have $|\langle h,\mathfrak{q}_{t_n,c_*}\rangle| \leq \frac{(2g+1) \|h\|}{(n+1)}$, so that the fact that $H$ is uniformly continuous implies that:
$$
\sum_{q_n\mbox{\scriptsize
dominant}} H(\langle h,\mathfrak{p}_{q_n}\rangle) 
=
\left(\frac{2}{2^gg!}
\sum_{(t_n,c_*)\in\mathcal{W}_{g,n}}
H
(\langle h,\mathfrak{p}_{t_n}\rangle)
\right)
\Big(1+o
\left(1\right)
\Big).
$$
On the other hand, using the second statement of Lemma~\ref{lemma:enumskeleton} and the fact that $H$ is bounded, we obtain:
$$
\sum_{(t_n,c_*)\in\mathcal{W}_{g,n}}
H(\langle h,\mathfrak{p}_{t_n}\rangle) 
= \frac{1}{6^g}
\sum_{(t_n,c_*)\in\mathcal{R}_{g,n}}
H(\langle h,\mathfrak{p}_{t_n}\rangle)
\left(1+O(n^{-1/4})\right)
$$
Moreover, each labelled tree $t$ corresponds to exactly $\left(\sum_k X_t(k)^3\right)^g$ distinct
elements of $\mathcal{R}_{g}$, so that:
\begin{eqnarray*}
\sum_{(t_n,c_*)\in\mathcal{R}_{g,n}}
H(\langle h,\mathfrak{p}_{t_n}\rangle)
=\sum_{t_n\mbox{ \scriptsize labelled}}\left(
\sum_k X_{t_n}(k)^3
\right)^g
H(\langle h,\mathfrak{p}_{t_n}\rangle).
\end{eqnarray*}
Putting everything together gives:
$$
\mathbb{E}\left[H(\langle h,\mathfrak{p}_{\mathfrak{q}_n}\rangle)\right]
=\frac{2(1+o(1))}{12^g g! |\mathcal{Q}^\bullet_{g,n}|}
\sum_{t_n\mbox{ \scriptsize labelled}}
\left(
\sum_k X_{t_n}(k)^3
\right)^g
H(\langle h,\mathfrak{p}_{t_n}\rangle)
$$
Now, from the expression of $t_g$ given in Theorem~\ref{thm:ISE}, we have 
$|\mathcal{Q}^\bullet_{g,n}|\sim \frac{2\cdot\gamma^{2g}\cdot Z_g}{12^gg!\sqrt{\pi}} n^{\frac{5g-3}{2}}12^n$, whereas the total number of labelled trees is equivalent to $\frac{2}{\sqrt{\pi}}n^{-\frac{3}{2}}12^n$ (this corresponds to $g=0$). Consequently, we can rewrite:
\begin{eqnarray}
\label{eq:onlytrees}
\mathbb{E}\left[H(\langle h,\mathfrak{p}_{\mathfrak{q}_n}\rangle)\right]
=\frac{1}{Z_g}
\frac{\sum_{t_n}\left(
\frac{1}{\gamma^2n^{5/2}}\sum_k X_{t_n}(k)^3
\right)^g
H(\langle h,\mathfrak{p}_{t_n}\rangle)}{|\{\mbox{rooted labelled trees with } n
\mbox{ edges}\}|} (1+o(1))
\end{eqnarray}

\subsubsection*{\bf step 2: the convergence}

We let, for all $n\geq 0$, $\mathfrak{t}_n$ be a rooted labelled tree chosen
uniformly at random among rooted labelled trees with $n$ edges, and we let as before:
$$
f_n(x) = \frac{X_{\mathfrak{t}_n}([\gamma^{-1}n^{1/4}x])}{\gamma n^{3/4}}, \ \ x\in\mathbb{R}.
$$ 
We also let $W_n=\frac{1}{\gamma^2 n^{5/2}}\sum X_n(k)^3=\int_{-\infty}^\infty f_n(x)^3dx$.
 Observe that
Equation~\ref{eq:onlytrees} rewrites:
\begin{eqnarray}\label{eq:onlytreesexp}
\mathbb{E}\left[H(\langle h,\mathfrak{p}_{\mathfrak{q}_n}\rangle)\right]
= \lim_n \ \frac{1}{Z_g}
\mathbb{E}\left[{W_n}^g H(\langle h,\mathfrak{p}_{t_n}\rangle)\right].
\end{eqnarray}
We let $[l_n,r_n]$ be the support of $f_n$ (i.e. $[l_n,r_n]$ is the intersection of all the real intervals outside which $f_n$ is identically $0$).
It is known (see \cite{ChassaingSc}) that $l_n$ and $r_n$ converge in law to
$r$ and $l$. However, we need a little more, namely to control the joint
convergence of $l_n$, $r_n$ and $f_n$. We have:
\begin{lemma}
\label{lemma:probabilityspace}
There exists a probability space $(\tilde \Omega,\tilde{\mathcal{F}},\tilde P)$, and random variables 
$\tilde f$ and $(\tilde f_n)_{n\geq 0}$ on that space, such that $\tilde f =_d f_\ISE$, $\tilde f_n=_d f_n$ and that if 
$[\tilde l_n,\tilde r_n]$ and $[\tilde l,\tilde r]$ denote respectively  the support of  $\tilde f_n$ and $\tilde f$, then the triple $(\tilde f_n, \tilde l_n, \tilde r_n)$ converges almost surely to $(\tilde f,\tilde l,\tilde r)$ in $\mathcal{D}_0(\mathbb{R})\times\mathbb{R}\times\mathbb{R}$.
\end{lemma}
We postpone the proof of the lemma, and we continue the proof of the theorem.

\noindent $\bullet$ First, we define the probability measure: 
$$
\tilde{\mathfrak{p}}_{n} :=
\frac{1}{n+1} \sum_{k\geq 0} \gamma n^{3/4} \tilde f _n \left(\frac{k}{\gamma^{-1} n^{1/4}} + \tilde l_n\right)
\delta_{\gamma n^{-1/4}(k+1)}.
$$
Then clearly, the $4$-tuple $(\tilde f_n, \tilde l_n, \tilde r_n, \tilde{\mathfrak{p}}_{n})$ has the same law as 
$(f_n,l_n,r_n,\mathfrak{p}_{\mathfrak{t}_n})$.
Moreover, if we define the probability measure $\overline {\tilde \mu}$ by $\overline {\tilde \mu}(dx) = \tilde f(x+\tilde l) dx$, then $\tilde{\mathfrak{p}}_n$ converges almost surely to $\overline{\tilde \mu}$ (in the sense of weak convergence).
Indeed, if $u$ is a bounded and uniformly continuous real function, one has:
\begin{eqnarray*}
\langle u, \tilde{\mathfrak{p}}_n \rangle &=& 
\frac{\gamma n^{3/4}}{n+1} \sum_{k\geq 0} 
\tilde f_n \left(\frac{k}{\gamma^{-1}n^{1/4}}+\tilde l_n\right) u\left(\frac{k+1}{\gamma^{-1}n^{1/4}}\right) \\
&=&\frac{n}{n+1} \int_\mathbb{R} 
\tilde f_n(x+\tilde l_n)  u\left(\frac{[\gamma^{-1}n^{1/4}x]+1}{\gamma^{-1}n^{1/4}}\right)dx \\
&=&\frac{n}{n+1} \int_\mathbb{R} 
\tilde f_n(x)  u\left(\frac{[\gamma^{-1}n^{1/4}(x-\tilde l_n)]+1}{\gamma^{-1}n^{1/4}}\right)dx.
\end{eqnarray*}
Since $u$ is uniformly continuous, and since $\tilde f_n \rightarrow \tilde f$ and $\tilde l_n \rightarrow \tilde l$, almost surely, this last quantity tends almost surely to $\int_\mathbb{R} \tilde f(x) u(x-\tilde l) dx=\langle u, \overline{\tilde \mu}\rangle$. Hence, taking $u$ along a countable dense subset of real bounded continuous functions shows that $\tilde{\mathfrak{p}}_n$ converges a.s. to $\overline{\tilde \mu}$, as claimed.
\\
\noindent $\bullet$ 
We let $\tilde W_n:= \int_\mathbb{R} \tilde f_n (x)^3 dx$ and $\tilde W:= \int_\mathbb{R} \tilde f (x)^3 dx$, so that we have 
$\mathbb{E}\left[W_n^g H(\langle h, \mathfrak{p}_{\mathfrak{t}_n}\rangle)\right]=
\mathbb{E}\left[\tilde W_n^g H(\langle h, \tilde{\mathfrak{p}}_{n}\rangle)\right]$ and 
$\mathbb{E}\left[W^g H(\langle h, \overline{\mu}_\ISE\rangle)\right]=
\mathbb{E}\left[\tilde W^g H(\langle h, \overline{\tilde \mu}\rangle)\right]$.
We now write:
\begin{eqnarray*}
\lefteqn{
\left|\mathbb{E}\left[\tilde W_n^g H(\langle h, \tilde{\mathfrak{p}}_{n}\rangle)\right]-
\mathbb{E}\left[\tilde W^g H(\langle h, \overline{\tilde \mu}\rangle)\right]\right|
} 
\\ &  &
\leq 
\underbrace{
\left|\mathbb{E}\left[ (\tilde W_n^g -\tilde W^g)H(\langle h, \tilde{\mathfrak{p}}_{n}\rangle)\right]
\right|}_{T_1}
+
\underbrace{
\left|\mathbb{E}\left[\tilde W^g (H(\langle h, \tilde{\mathfrak{p}}_{n}\rangle)
-H(\langle h, \overline{\tilde \mu}\rangle))
\right]
\right|}_{T_2}.
\end{eqnarray*}
In order to bound the first term, we write $|T_1|\leq \|H\| \mathbb{E}\left[|\tilde W_n^g -\tilde W^g|\right]$.
This last quantity tends to $0$ from the two facts that
that $\mathbb{E}\left[\tilde W_n^g\right]\rightarrow\mathbb{E}\left[\tilde W^g\right]$, and that $\tilde W_n$ converges almost surely to $\tilde W$ (which is a direct consequence of the a.s. convergence of the triple $(\tilde f_n,\tilde l_n,\tilde r_n)$). Indeed, this follows from applying Fatou's lemma to the nonnegative random variable 
$\tilde W_n^g+\tilde W^g-|\tilde W_n^g-\tilde W^g|$ (this argument is also known as Scheffe's lemma).

As for the second term $T_2$, it tends to $0$ by dominated convergence. Indeed, $H(\langle h, \tilde{\mathfrak{p}}_{n}\rangle)-H(\langle h, \overline{\tilde \mu}\rangle)$ tends to $0$ almost surely since $\tilde{\mathfrak{p}}_n$ converges a.s. weakly to $\overline{\tilde \mu}$, and moreover the integrand is bounded by $2\|H\|\tilde W^g$, which we know is integrable.

Therefore we have proved that:
$$\left|
\mathbb{E}\left[\tilde W^g H(\langle h,\overline{\tilde \mu}\rangle) \right]-
\mathbb{E}\left[\tilde W_n^g H(\langle h,\tilde{\mathfrak{p}}_n\rangle) \right]\right|
\longrightarrow 0,
$$
which, together with Equation~\ref{eq:onlytreesexp}, gives Equation~\ref{eq:convergenceM1}. This proves (up to Lemma~\ref{lemma:probabilityspace}) the convergence of $\mathcal{L}^g_n$ to $\mathcal{L}^g$. The proof of convergence of the radius goes along the same lines.

\subsection{Proof of Lemma~\ref{lemma:probabilityspace}}

From \cite{MBM-Janson}, we know that $f_n$ converges to $f_\ISE$ in distribution, and from \cite{ChassaingSc}, that $l_n$ and $r_n$ converge respectively to $l$ and $r$, in distribution. It follows that the sequence of triples $(f_n,l_n,r_n)_{n\geq 1}$ is \emph{tight} in the space $\mathcal{D}_0(\mathbb{R})\times \mathbb{R} \times \mathbb{R}$.
Up to extraction, we can therefore assume that this sequence converges in distribution to some triple $(f,l',r')$, such that 
$f=_d f_\ISE$, $l'=_d l$, $r' =_d r$. By Skorokhod's representation theorem%
\footnote{The space $\mathcal{D}_0(\mathbb{R})$ equipped with the uniform norm is not separable, but the law of 
$f_\ISE$ is supported on the subset $\mathcal{C}_K(\mathbb{R})$ of continuous functions with compact support, which is. Hence Skorokhod's representation theorem applies.}%
, there exists a probability space, and random variables $\hat f_n, \hat f, \hat l', \hat r'$, such that $\hat f_n =_d f_n$, $\hat f =_d f_\ISE$, $\hat l' =_d l$, $\hat r' =_d r$, and such that if $[\hat l_n, \hat r_n]$ denotes the support of $\hat f_n$, we have the almost sure convergence $(\hat f_n, \hat l_n, \hat r_n) \rightarrow (\hat f, \hat l', \hat r')$, along the extraction mentionned above.
 
Now, let $[\hat l,\hat r]$ be the support of $\hat f$. The fact that $\hat f_n\rightarrow \hat f$ uniformly implies that (always along the same extraction):
$$
\limsup_n \hat l_n \leq \hat l \leq \hat r \leq \liminf_n \hat r_n, 
$$
which, since $(\hat l_n,\hat r_n)\rightarrow (\hat r', \hat l')$, gives:
$\hat l' \leq \hat l \leq \hat r \leq \hat r'$.
Now, since $\hat r$ and $\hat r'$ have the same distribution, we have for every $M>0$, 
$\mathbb{E}\left[\hat r'\wedge M - \hat r \wedge M\right]=0$. The quantity in the expectation being nonnegative, it follows that $\hat r'\wedge M = \hat r \wedge M$ almost surely, so that by letting $M$ tend to infinity we obtain that $\hat r=\hat r'$ almost surely. Similarly we have $\hat l=\hat l'$ almost surely. It follows that, along the aforementionned extraction, the \emph{triple} $(\hat f_n,\hat l_n,\hat r_n)$ converges in distribution to $(\hat f,\hat l,\hat r)$, or equivalently that 
$(f_n,l_n,r_n)$ converges in distribution to $(f_\ISE, l, r)$.

But, since the limit $(f_\ISE,l,r)$ does not depend on the extraction, it follows that this convergence actually holds when $n$ tends to infinity, without considering extractions anymore:
$$
(f_n,l_n,r_n) \stackrel{d}{\longrightarrow} (f_\ISE,l,r), \ \textrm{ when }n \rightarrow \infty$$
The lemma follows by a last application of Skorokhod's representation theorem.

\section*{A concluding remark}

It is known that a (sort of)  generating series of the $t_g$'s satisfies a
remarquable differential equation of Painlev\'e-I type,
which enables in particular to compute very easily
$t_g$ at any order. Precisely, if we set:
$$
u(y) = -\sum_{g\geq 0} 4^{g-1}\Gamma(\frac{5g-1}{2}) t_g
y^{\frac{1-5g}{2}} $$
then (see~\cite{LaZv}, page 201) $u$ satisfies the Painlev\'e-I
equation: \begin{eqnarray}
\label{eq:painleve}
y=u(y)^2+u''(y).
\end{eqnarray}
Now, observe that from Theorem~\ref{thm:ISE}, we can express $u$
as a functional transform of the random variable $W(\mu_\ISE)$:
$$
u(y)=-\frac{y^{1/2}}{2\sqrt{\pi}}\mathbb{E}
\left[
\sum_{g\geq 0} \frac{\Gamma(\frac{5g-1}{2})}{g!}
\left(\frac{1}{\sqrt{2}y^{5/2}}W(\mu_{\ISE})\right)^g
\right].
$$
We hope that this opens the way to a new derivation of
Equation~\ref{eq:painleve}, via the theory of superprocesses. This would be an
important achievement of the enumerative theory of maps via labelled
trees.

\section*{Acknowledgements.} 
We thank an anonymous referee who greatly helped improving the probabilistic part of the paper. In particular, our first proof of Lemma~\ref{lemma:probabilityspace} was less conceptual than the one (s)he suggested, which is reproduced here. (S)he also helped fixing some problems in the proofs of Theorems~\ref{thm:ISE} and~\ref{thm:profile}.

Thanks also to Jean-Fran\c cois Marckert  and to my advisor Gilles Schaeffer, for stimulating discussions.

\bibliographystyle{alpha}
\bibliography{UnicellularOpening}

\label{sec:biblio}

\end{document}